\documentclass[10pt,reqno]{amsart}

\usepackage{amsthm, mathrsfs, amsmath, amstext, amsxtra, amsfonts, dsfont, amssymb, color}
\usepackage{lmodern}
\definecolor{green_dark}{rgb}{0,0.6,0}
\usepackage[colorlinks, linkcolor=red, citecolor=blue, urlcolor=green_dark, pagebackref, hypertexnames=false]{hyperref}
\usepackage{caption}
\usepackage{float}
\usepackage{multirow}

\newcommand{\N}{\mathbb N}
\newcommand{\Z}{\mathbb Z}

\newcommand{\R}{\mathbb R}
\newcommand{\C}{\mathbb C}

\newcommand{\Gact} {\gamma_{\text{c}}}
\newcommand{\Gace} {\gamma_{\emph{c}}}

\newcommand{\re}[1]{\mbox{Re} \ #1} 
\newcommand{\im}[1]{\mbox{Im} \ #1} 
\newcommand{\scal}[1]{\left\langle #1 \right\rangle} 

\newcommand{\defendproof}{\hfill $\Box$} 

\newtheorem{theorem}{Theorem}[section]
\newtheorem{defi}[theorem]{Definition}
\newtheorem{lem}[theorem]{Lemma} 
\newtheorem{prop}[theorem]{Proposition}
\newtheorem{coro}[theorem]{Corollary} 
\theoremstyle{definition}
\newtheorem{rem}[theorem]{Remark}

\title[Global existence scattering nonlinear fourth-order Schr\"odinger]{Global existence and scattering for a class of nonlinear fourth-order Schr\"odinger equation below the energy space} 

\author[V. D. Dinh]{Van Duong Dinh}
\address[V. D. Dinh]{Institut de Math\'ematiques de Toulouse UMR5219, Universit\'e Toulouse CNRS, 31062 Toulouse Cedex 9, France}
\email{dinhvan.duong@math.univ-toulouse.fr}


\begin{document}

\maketitle
\begin{abstract}
In this paper, we consider a class of nonlinear fourth-order Schr\"odinger equation, namely
\[
\left\{
\begin{array}{rcl}
i\partial_t u +\Delta^2 u &=&-|u|^{\nu-1} u, \quad 1+ \frac{8}{d}<\nu <1+\frac{8}{d-4},\\
u(0)&=&u_0 \in H^\gamma(\R^d), \quad 5 \leq d \leq 11.
\end{array}
\right. 
\]
Using the $I$-method combined with the interaction Morawetz inequality, we establish the global well-posedness and scattering in $H^\gamma(\R^d)$ with $\gamma(d,\nu)<\gamma<2$ for some value $\gamma(d,\nu)>0$.    
\end{abstract}


\section{Introduction}
\setcounter{equation}{0}
Consider the following nonlinear fourth-order Schr\"odinger equation
\begin{align}
\left\{
\begin{array}{rcl}
i\partial_t u(t,x) + \Delta^2 u(t,x) &=& -(|u|^{\nu-1} u)(t,x), \quad t\in \R, x\in \R^d, \\
u(0,x) &=& u_0(x) \in H^\gamma(\R^d), 
\end{array}
\right.
\tag{NL4S}
\end{align}
where $u(t,x)$ is a complex valued function in $\R \times \R^d, d\geq 5$. The nonlinear exponent $\nu$ is assumed to be mass-supercritical, i.e $\nu>1+\frac{8}{d}$ and energy-subcritical, i.e. $\nu<1+\frac{8}{d-4}$. 
The regularity exponent $\gamma$ is assumed to satisfy $0<\gamma<2$.\newline
\indent The fourth-order Schr\"odinger equation was introduced by Karpman \cite{Karpman} and Karpman-Shagalov \cite{KarpmanShagalov} to take into account the role of small fourth-order dispersion terms in the propagation of intense laser beams in a bulk medium with Kerr nonlinearity. Such a fourth-order Schr\"odinger equation is of the form
\begin{align}
i\partial_t u + \Delta^2 u+ \varepsilon \Delta u + \mu|u|^{\nu-1}u=0, \quad u(0)=u_0, \label{general fourth order schrodinger}
\end{align}
where $\varepsilon \in \{0,\pm 1\}, \mu \in \{\pm 1\}$ and $\nu>1$. We note that (NL4S) is a special case of $(\ref{general fourth order schrodinger})$ by taking $\varepsilon=0$ and $\mu=1$. The nonlinear fourth-order Schr\"odinger equation $(\ref{general fourth order schrodinger})$ has attracted a lot of interest in a past decay. The sharp dispersive estimates for the linear part of $(\ref{general fourth order schrodinger})$ were established in \cite{Ben-ArtziKochSaut}. The local well-posedness and the global well-posedness for $(\ref{general fourth order schrodinger})$ has been widely studied in \cite{Dinhfract, Dinhfourth, Dinhmass, Dinhblowup, Guo, GuoCui06, HaoHsiaoWang06, HaoHsiaoWang07, HuoJia, MiaoXuZhao09, MiaoXuZhao11, MiaoWuZhang, MiaoZhang, Pausader, Pausadercubic, PausaderShao10} and references therein. \newline
\indent The (NL4S) enjoys a natural scaling invariance, that is if we set for $\lambda>0$
\begin{align}
u_\lambda(t,x):= \lambda^{-\frac{4}{\nu-1}} u(\lambda^{-4}t, \lambda^{-1}x), \label{scaling definition}
\end{align}
then for $T\in(0,+\infty]$,
\[
u \text{ solves (NL4S) on } (-T,T) \Longleftrightarrow u_\lambda \text{ solves (NL4S) on } (-\lambda^4T, \lambda^4T).  
\]
We define the critical regularity exponent for (NL4S) by
\begin{align}
\Gact:=\frac{d}{2}-\frac{4}{\nu-1}. \label{critical regularity exponent}
\end{align}
The (NL4S) is known (see \cite{Dinhfract} or \cite{Dinhfourth}) to be locally well-posed in $H^\gamma(\R^d)$ with $\gamma \geq \max\{0,\Gact\}$ satisfying for $\nu$ is not an odd integer, 
\begin{align}
\lceil \gamma \rceil \leq \nu. \label{regularity condition}
\end{align} 
Here $\lceil \gamma \rceil$ is the smallest integer greater than or equal to $\gamma$. This condition ensures the nonlinearity to have enough regularity. In the sub-critical regime, i.e. $\gamma>\Gact$, the time of existence depends only on the $H^\gamma$-norm of initial data. Moreover, the local solution enjoys mass conservation, i.e.
\[
M(u(t)):=\|u(t)\|^2_{L^2(\R^d)} = \|u_0\|^2_{L^2(\R^d)},
\] 
and $H^2$-solution has conserved energy, i.e.
\[
E(u(t)):=\int_{\R^d} \frac{1}{2}|\Delta u(t,x)|^2 + \frac{1}{\nu+1}|u(t,x)|^{\nu+1} dx = E(u_0).
\]
The persistence of regularity (see \cite{Dinhfourth}) combined with the conservations of mass and energy yield the global well-posedness for (NL4S) in $H^\gamma(\R^d)$ with $\gamma\geq 2$ satisfying for $\nu$ is not an odd integer, $(\ref{regularity condition})$. In the critical regime, i.e. $\gamma=\Gact$, one also has (see \cite{Dinhfract} or \cite{Dinhfourth}) the local well-posedness for (NL4S) but the time of existence depends not only on the $H^\gamma$-norm of initial data but also on its profile. Moreover, for small initial data, the (NL4S) is globally well-posed, and the solution is scattering. \newline
\indent The main goal of this paper is to show the global well-posedness and scattering for the nonlinear fourth-order Schr\"odinger equation (NL4S) below the energy space. Our arguments are based on the combination of the $I$-method and the interaction Morawetz inequality which are similar to those of \cite{VisanZhang}. However, there are some difficulties due to the high-order dispersion term $\Delta^2u$. Moreover, in order to successfully establish the almost conservation law, we need the nonlinearity to have at least two orders of derivatives. This leads to the restriction in spatial space of dimensions $5\leq d \leq 11$. \newline
\indent Before stating our main result, let us recall some known results concerning the global existence below the energy space for the nonlinear fourth-order Schr\"odinger equation. To our knowledge, Guo in \cite{Guo} gave a first answer to this problem. In \cite{Guo}, the author considered $(\ref{general fourth order schrodinger})$ with $\nu-1=2m, m\in \N$ satisfying $4<md<4m+2$, and established the global existence in $H^\gamma(\R^d)$ with $1+\frac{md-9+\sqrt{(4m-md+7)^2+16}}{4m}<\gamma<2$. The proof is based on the $I$-method which is a modification of the one invented by $I$-Team \cite{I-teamalmost} in the context of nonlinear Schr\"odinger equation. Later, Miao-Wu-Zhang in \cite{MiaoWuZhang} studied the defocusing cubic fourth-order Schr\"odinger equation, i.e. $\nu=3$ in (NL4S), and proved the global well-posedness and scattering in $H^\gamma(\R^d)$ with $\gamma(d)<\gamma<2$ where $\gamma(5)=\frac{16}{11}, \gamma(6)=\frac{16}{9}$ and $\gamma(7)=\frac{45}{23}$. The proof relies on the combination of the $I$-method and a new interaction Morawetz inequality. Recently, in \cite{Dinhmass} the author considered the defocusing cubic higher-order Schr\"odinger equation including the cubic fourth-order Schr\"odinger equation, and showed that the (NL4S) with $\nu=3$ is globally well-posed in $H^\gamma(\R^4)$ with $\frac{60}{53}<\gamma<2$. The argument makes use of the $I$-method and the bilinear Strichartz estimate. The analysis is carried out in Bourgain spaces $X^{\gamma,b}$ which is similar to those in \cite{I-teamalmost}. In the above considerations, the nonlinearity is algebraic, i.e. $\nu$ is an odd integer. This allows to write the commutator between the $I$-operator and the nonlinearity explicitly by means of the Fourier transform, and then carefully control the frequency interactions using multi-linear analysis. When one considers the nonlinear fourth-order Schr\"odinger equation (NL4S) with $\nu>1$ is not an odd integer, this method does not work. We thus rely purely on Strichartz and interaction Morawetz estimates. \newline
\indent Let us now introduce some notations. 
\begin{align}
\gamma(d,\nu):=\max \{\gamma_1(d,\nu), \gamma_2(d,\nu), \gamma_3(d,\nu), \gamma_4(d,\nu)\}, \label{define gamma d nu}
\end{align}
where
\begin{align*}
\gamma_1(d,\nu) &:=\frac{3}{2}+\frac{\Gact}{4}, \\
\gamma_2(d,\nu) &:= 4-\nu, \\
\gamma_3(d,\nu) &:=\frac{2}{\nu-1} +\frac{(\nu-2)\Gact}{\nu-1}, \\
\gamma_4(d,\nu) &:=\min_{\sigma \in (0,\sigma_0]} \gamma(d,\nu,\sigma).
\end{align*}
Here $\sigma_0$ satisfies 
\begin{align}
\left\{ 
\begin{array}{rcl}
2\sigma_0(16-(\nu-1)(d+4))&<&(d-5)(d(\nu-1)-8), \\
2\sigma_0(\nu-3)&\leq& d-5, \\
\sigma_0 &\leq& \gamma,
\end{array}
\right. \label{condition sigma_0}
\end{align}
and $\gamma(d,\nu,\sigma)$ is the (large if there are two) root of the equation
\[
\Gact(2-\gamma)(d-5+(8-d)\sigma) = \min \left\{ \gamma-1-\frac{\Gact}{2}, \nu-2, (\nu-2) (\gamma-\Gact) \right\} (\gamma-\Gact)\sigma.
\]
The main result of this paper is the following:
\begin{theorem} \label{theorem global existence scattering}
Let $5\leq d\leq 11$. The initial value problem \emph{(NL4S)} is globally well-posed in $H^\gamma(\R^d)$ for any $\gamma(d,\nu)<\gamma<2$, and the global solution $u$ enjoys the following uniform bound
\[
\|u\|_{L^\infty(\R, H^\gamma(\R^d))} \leq C(\|u_0\|_{H^\gamma(\R^d)}).
\]
Moreover, the solution is scattering, i.e. there exist unique $u^\pm_0\in H^\gamma(\R^d)$ such that
\[
\lim_{t\rightarrow \pm\infty} \|u(t)-e^{it\Delta^2} u^\pm_0\|_{H^\gamma(\R^d)} =0.
\]
\end{theorem}  
We record in the table below some best known results, and compare them with our ones. As in the table, our results are not as good as the best known results when $\nu$ is an odd integer. But our method allows to treat the non-algebraic nonlinearity.  
\begin{table}[H]
\centering
\renewcommand{\arraystretch}{1.5}
\begin{tabular}{|c|c|c|c|c|}
\hline
$\nu$ & $d$ & $\Gact$ & $\gamma(d,\nu)$ (best known results) & $\gamma(d,\nu)$ (our results) \\
\hline
\multirow{3}{*}{3} & 5 & $\frac{1}{2}$ & $\frac{16}{11}\approx 1.4545$ (see \cite{MiaoWuZhang}) & $1.6711$ \\
\cline{2-5}
 & 6 & 1 & $\frac{16}{9}\approx 1.7777$ (see \cite{MiaoWuZhang}) & $1.8719$ \\
\cline{2-5}
 & 7 & $\frac{3}{2}$ & $\frac{45}{23} \approx 1.9565 $ (see \cite{MiaoWuZhang}) & $1.9665$ \\
\hline
\multirow{2}{*}{4} & 5 & $\frac{7}{6}$ & - & $1.9257$ \\
\cline{2-5}
& 6 & $\frac{5}{3}$ & - & $1.9922$ \\ 
\hline
\end{tabular}
\caption{Our results compare with best known results.}
\label{table:compare}
\end{table}
The proof of the above result is based on two main ingredients: the $I$-method and the interaction Morawetz inequality, which are similar to those given in \cite{VisanZhang}. The $I$-method for the fourth-order Schr\"odinger equation is a modification of the one introduced by $I$-Team in \cite{I-teamalmost}. This method is very useful for treating the nonlinear dispersive equation at low regularity, i.e. below energy space. The idea is to replace the non-conserved energy $E(u)$ when $\gamma<2$ by an ``almost conserved'' variance $E(Iu)$ with $I$ a smoothing operator which is the identity at low frequency and behaves like a fractional integral operator of order $2-\gamma$ at high frequency. Since $Iu$ is not a solution of (NL4S), we may expect an energy increment. The key is to show that the modified energy $E(Iu)$ is an ``almost conserved'' quantity in the sense that the time derivative of $E(Iu)$ decays with respect to a large parameter $N$ (see Section $\ref{section preliminaries}$ for the definition of $I$ and $N$). To do so, we need delicate estimates on the commutator between the $I$-operator and the nonlinearity. When the nonlinearity is algebraic, we can use the Fourier transform to write this commutator explicitly, and then carefully control the frequency interactions. Once the nonlinearity is no longer algebraic, this method fails. In order to treat this case, we take the advantage of Strichartz estimate with a gain of derivatives $(\ref{strichartz estimate biharmonic 4order})$. Thanks to this Strichartz estimate, we are able to apply the technique given in \cite{VisanZhang} to control the commutator. Of course, this technique is not as good as the Fourier transform technique when the nonlinearity is algebraic, but it is more robust and allows us to treat the non-algebraic nonlinearity. The interaction Morawetz inequality for the nonlinear fourth-order Schr\"odinger equation was first introduced in \cite{Pausadercubic} for $d\geq 7$. Then, it was extended for $d\geq 5$ in \cite{MiaoWuZhang}. Using this interaction Morawetz inequality and the interpolation argument together with the Sobolev embedding, we have for any compact interval $J$ and $0<\sigma \leq \gamma$,
\begin{align}
\|u\|_{M^\sigma(J)}:=\|u\|_{L^{\frac{d-5+4\sigma}{\sigma}}_t(J, L^{\frac{2(d-5+4\sigma)}{d-5+2\sigma}}_x)} \lesssim \Big(\|u_0\|_{L^2_x} \|u\|_{L^\infty_t(J,\dot{H}^{\frac{1}{2}}_x)}\Big)^{\frac{2\sigma}{d-5+4\sigma}} \|u\|_{L^\infty_t(J,\dot{H}^\sigma_x)}^{\frac{d-5}{d-5+4\sigma}}. \label{morawetz norm introduction}
\end{align}
As a byproduct of the Strichartz estimates and $I$-method, we show the ``almost conservation law'' for (NL4S), that is if $u \in L^\infty(J, \mathscr{S}(\R^d))$ is a solution to (NL4S) on a time interval $J=[0,T]$, and satisfies $\|Iu_0\|_{H^2_x}\leq 1$ and if $u$ satisfies in addition the a priori bound $\|u\|_{M^\sigma(J)} \leq \mu$ for some small constant $\mu>0$, then
\[
\sup_{t\in[0,T]}|E(Iu(t))-E(Iu_0)| \lesssim N^{-(2-\gamma+\delta)},
\] 
for some $\delta>0$. \newline
\indent We now give an outline of the proof. Let $u$ be a global in time solution to (NL4S) with initial data $u_0 \in C^\infty_0(\R^d)$. Our goal to to show the uniform bounds
\begin{align}
\|u\|_{M^\sigma(\R)} &\leq C(\|u_0\|_{H^\gamma_x}), \label{global bound morawetz norm introduction} \\
\|u\|_{L^\infty_t(\R, H^\gamma_x)} &\leq C(\|u_0\|_{H^\gamma_x}), \label{global bound sobolev norm introduction}
\end{align}
Thanks to $(\ref{global bound sobolev norm introduction})$, the global existence follows immediately by a standard density argument. Since $E(Iu_0)$ is not necessarily small, we will use the scaling $(\ref{scaling definition})$ to make $E(Iu_\lambda(0))$ small in order to apply the ``almost conservation law''. 
By choosing 
\begin{align}
\lambda \sim N^{\frac{2-\gamma}{\gamma-\Gact}}, \label{lambda introduction}
\end{align}
and using some harmonic analysis, we can make $E(Iu_\lambda(0)) \leq \frac{1}{4}$. We will show that there exists an absolute constant $C$ such that
\begin{align}
\|u_\lambda\|_{M^\sigma(\R)} \leq C\lambda^{\Gact+\frac{\sigma(4-d)\Gact}{2(d-5+4\sigma)}}. \label{global bound scaling morawetz norm introduction}
\end{align}
We then obtain $(\ref{global bound morawetz norm introduction})$ by undoing the scaling. In order to prove $(\ref{global bound scaling morawetz norm introduction})$, we perform a bootstrap argument. Note that $(\ref{global bound scaling morawetz norm introduction})$ is equivalent to
\[
\|u_\lambda\|_{M^\sigma([0,t])} \leq C\lambda^{\Gact+\frac{\sigma(4-d)\Gact}{2(d-5+4\sigma)}}, \quad \forall t\in \R.
\]
Assume by contraction, it is not so. Since $\|u_\lambda\|_{M^\sigma([0,t])}$ is a continuous function in $t$, there exists $T>0$ so that
\begin{align}
\|u_\lambda\|_{M^\sigma([0,T])} &> C \lambda^{\Gact+\frac{\sigma(4-d)\Gact}{2(d-5+4\sigma)}}, \label{assumption 1 introduction} \\
\|u_\lambda\|_{M^\sigma([0,T])}&\leq 2C \lambda^{\Gact+\frac{\sigma(4-d)\Gact}{2(d-5+4\sigma)}}. \label{assumption 2 introduction} 
\end{align}
Using $(\ref{assumption 2 introduction})$, we can split $[0,T]$ into $L$ subintervals $J_k, k=1,...,L$ so that 
\[
\|u_\lambda\|_{M^\sigma(J_k)} \leq \mu.
\] 
The number $L$ must satisfy
\begin{align}
L \sim \lambda^{\frac{\Gact(d-5+(8-d)\sigma)}{\sigma}}. \label{L introduction}
\end{align}
We thus can apply the ``almost conservation law'' to get
\[
\sup_{[0,T]} E(Iu_\lambda(t)) \leq E(Iu_\lambda(0)) + N^{-(2-\gamma+\delta)} L.
\]
Since $E(Iu_\lambda(0))\leq \frac{1}{4}$, we need 
\begin{align}
N^{-(2-\gamma+\delta)} L \ll \frac{1}{4} \label{smallness introduction}
\end{align}
in order to guarantee $E(Iu_\lambda(t))\leq 1$ for all $t\in [0,T]$. Combining $(\ref{lambda introduction}), (\ref{L introduction})$ and $(\ref{smallness introduction})$, we get a condition on $\gamma$. Next, by $(\ref{morawetz norm introduction})$ and some harmonic analysis, we have
\begin{multline*}
\|u_\lambda\|_{M^\sigma([0,T])} \leq C(\|u_0\|_{L^2_x}) \lambda^{\Gact+\frac{\sigma(4-d)\Gact}{2(d-5+4\sigma)}} \sup_{[0,T]} \Big(\|Iu_\lambda(t)\|^{\frac{1}{4}}_{\dot{H}^2_x} + \|Iu_\lambda(t)\|^{\frac{1}{2\gamma}}_{\dot{H}^2_x} \Big)^{\frac{2\sigma}{d-5+4\sigma}} \\
\times \sup_{[0,T]} \Big(\|Iu_\lambda(t)\|^{\frac{\sigma}{2}}_{\dot{H}^2_x} + \|Iu_\lambda(t)\|^{\frac{\sigma}{\gamma}}_{\dot{H}^2_x}  \Big)^{\frac{d-5}{d-5+4\sigma}}.
\end{multline*}
Since $\|Iu_\lambda(t)\|^2_{\dot{H}^2_x} \lesssim E(Iu_\lambda(t)) \leq 1$ for all $t\in [0,T]$, we get
\[
\|u_\lambda\|_{M^\sigma([0,T])} \leq K\lambda^{\Gact+\frac{\sigma(4-d)\Gact}{2(d-5+4\sigma)}},
\]
for some constant $K>0$. This contradicts with $(\ref{assumption 1 introduction})$ by taking $C$ larger than 2K. We thus obtain $(\ref{global bound morawetz norm introduction})$ and also 
\[
E(Iu_\lambda(t)) \leq 1,\quad  \forall t\in [0,\infty).
\]
This also gives the uniform bound $(\ref{global bound sobolev norm introduction})$. In order to prove the scattering property, we will upgrade the uniform Morawetz bound $(\ref{global bound morawetz norm introduction})$ to the uniform Strichartz bound, namely
\[
\|u\|_{S^\gamma(\R)}:=\sup_{(p,q)\in B} \|\scal{\nabla}^\gamma u\|_{L^p_t(\R,L^q_x)} \leq C(\|u_0\|_{H^\gamma_x}).
\]
Here $(p,q) \in B$ means that $(p,q)$ is biharmonic admissible (see again Section \ref{section preliminaries} for the definition). With this uniform Strichartz bound, the scattering property follows by a standard argument. We refer the reader to Section $\ref{section global well-posedness scattering}$ for more details. \newline
\indent This paper is organized as follows. We firstly introduce some notations and recall some results related to our problem in Section $\ref{section preliminaries}$. In Section $\ref{section almost conservation law}$, we show the almost conservation law for the modified energy. Finally, we give the proof of our main result in Section $\ref{section global well-posedness scattering}$.
\section{Preliminaries} \label{section preliminaries}
\setcounter{equation}{0}
In the sequel, the notation $A \lesssim B$ denotes an estimate of the form $A\leq CB$ for some constant $C>0$. The notation $A \sim B$ means that $A \lesssim B$ and $B \lesssim A$. We write $A \ll B$ if $A \leq cB$ for some small constant $c>0$. We also use $\scal{a}:=1+|a|$. 
\subsection{Nonlinearity} \label{subsection nonlinearity}
Let $F(z):= |z|^{\nu-1} z$  be the function which defines the nonlinearity in (NL4S). The derivative of $F(z)$ is defined by
\[
F'(z):=(\partial_z F(z), \partial_{\overline{z}}F(z)),
\]
where
\[
\partial_z F(z)=\frac{\nu+1}{2} |z|^{\nu-1}, \quad \partial_{\overline{z}} F(z) = \frac{\nu-1}{2}|z|^{\nu-1} \frac{z}{\overline{z}}.
\]
We also define its norm as
\[
|F'(z)| := |\partial_zF(z)| + |\partial_{\overline{z}} F(z)|.
\]
It is clear that $|F'(z)| = O(|z|^{\nu-1})$. For a complex-valued function $u$, we have the following chain rule
\[
\partial_k F(u) = F'(u) \partial_k u, 
\] 
for $k \in \{1,\cdots, d\}$. In particular, 
\begin{align}
\nabla F(u)= F'(u) \nabla u. \label{chain rule}
\end{align}
In order to estimate the nonlinearity, we need to recall the following fractional chain rules.
\begin{lem}[\cite{ChristWeinstein}, \cite{KenigPonceVega}]\label{lem fractional chain rule C1}
Suppose that $G\in C^1(\C, \C)$, and $\alpha \in (0,1)$. Then for $1 <q \leq q_2 <\infty$ and $1<q_1 \leq \infty$ satisfying $\frac{1}{q}=\frac{1}{q_1}+\frac{1}{q_2}$, 
\[
\||\nabla|^\alpha G(u) \|_{L^q_x} \lesssim \|G'(u) \|_{L^{q_1}_x} \||\nabla|^\alpha u \|_{L^{q_2}_x}.
\]
\end{lem}
\begin{lem}[\cite{Visanthesis}]\label{lem fractional chain rule holder continuous}
Suppose that $G \in C^{0,\beta}(\C, \C), \beta \in (0,1)$. Then for every $0<\alpha<\beta, 1<q<\infty$, and $\frac{\alpha}{\beta}<\rho<1$,
\[
\||\nabla|^\alpha G(u)\|_{L^q_x} \lesssim \||u|^{\beta-\frac{\alpha}{\rho}}\|_{L^{q_1}_x} \||\nabla|^\rho u\|^{\frac{\alpha}{\rho}}_{L^{\frac{\alpha}{\rho}q_2}_x},
\]
provided $\frac{1}{q}=\frac{1}{q_1}+\frac{1}{q_2}$ and $\left(1-\frac{\alpha}{\beta \rho}\right)q_1>1$.
\end{lem}
The reader can find the proof of Lemma $\ref{lem fractional chain rule C1}$ in the case $1<q_1<\infty$ in \cite[Proposition 3.1]{ChristWeinstein} and \cite[Theorem A.6]{KenigPonceVega} when $q_1=\infty$. For the proof of Lemma $\ref{lem fractional chain rule holder continuous}$, we refer to \cite[Proposition A.1]{Visanthesis}. 
\subsection{Strichartz estimates} \label{subsection strichartz estimate}
Let $I \subset \R$ and $p, q \in [1,\infty]$. The Strichartz norm is defined as
\[
\|u\|_{L^p_t(I, L^q_x)} := \Big( \int_I \Big( \int_{\R^d} |u(t,x)|^q dx \Big)^{\frac{1}{q}} \Big)^{\frac{1}{p}}
\] 
with a usual modification when either $p$ or $q$ are infinity. When there is no risk of confusion, we write $L^p_t L^q_x$ instead of $L^p_t(I,L^q_x)$. When $p=q$, we also use $L^p_{t,x}$.
\begin{defi} \label{definition schrodinger admissible}
A pair $(p,q)$ is said to be \textbf{Schr\"odinger admissible}, for short $(p,q) \in S$, if 
\[
(p,q) \in [2,\infty]^2, \quad (p,q,d) \ne (2,\infty,2), \quad \frac{2}{p}+\frac{d}{q} \leq \frac{d}{2}.
\]
\end{defi}
We denote for $(p,q)\in [1,\infty]^2$,
\begin{align}
\gamma_{p,q}=\frac{d}{2}-\frac{d}{q}-\frac{4}{p}. \label{define gamma pq}
\end{align}
\begin{defi}\label{definition biharmonic admissible}
A pair $(p,q)$ is called \textbf{biharmonic admissible}, for short $(p,q)\in B$, if 
\[
(p,q) \in S, \quad \gamma_{p,q}=0.
\]
\end{defi}
\begin{prop}[Strichartz estimates for the fourth-order Schr\"odinger equation \cite{Dinhfract}] \label{prop strichartz}
Let $\gamma \in \R$ and $u$ be a (weak) solution to the linear fourth-order Schr\"odinger equation, namely
\[
u(t)= e^{it\Delta^2}u_0 + \int_0^t e^{i(t-s)\Delta^2} F(s) ds,
\]
for some data $u_0, F$. Then for all $(p,q)$ and $(a,b)$ Schr\"odinger admissible with $q<\infty$ and $b<\infty$,
\begin{align}
\||\nabla|^\gamma u\|_{L^p_t(\R, L^q_x)} \lesssim \||\nabla|^{\gamma+\gamma_{p,q}} u_0\|_{L^2_x} + \||\nabla|^{\gamma+\gamma_{p,q}-\gamma_{a',b'} -4} F\|_{L^{a'}_t(\R, L^{b'}_x)}. \label{generalized strichartz estimate}
\end{align}
Here $(a,a')$ and $(b,b')$ are conjugate pairs, and $\gamma_{p,q}, \gamma_{a',b'}$ are defined as in $(\ref{define gamma pq})$.
\end{prop}
The estimate $(\ref{generalized strichartz estimate})$ is exactly the one given in \cite{MiaoZhang}, \cite{Pausader} or \cite{Pausadercubic} where the author considered $(p,q)$ and $(a,b)$ are either sharp Schr\"odinger admissible, i.e.
\[
p, q \in [2,\infty]^2, \quad (p,q,d) \ne (2,\infty,2), \quad \frac{2}{p}+\frac{d}{q}=\frac{d}{2},
\]
or biharmonic admissible.
We refer the reader to \cite[Proposition 2.1]{Dinhfract} for the proof of Proposition $\ref{prop strichartz}$. The proof is based on the scaling technique instead of using a dedicate dispersive estimate of \cite{Ben-ArtziKochSaut} for the fundamental solution of the homogeneous fourth-order Schr\"odinger equation. \newline
\indent The following result is a direct consequence of $(\ref{generalized strichartz estimate})$. 
\begin{coro} \label{coro strichartz}
Let $\gamma \in \R$ and $u$ be a (weak) solution to the linear fourth-order Schr\"odinger equation for some data $u_0, F$. Then for all $(p,q)$ and $(a,b)$ biharmonic admissible satisfying $q<\infty$ and $b<\infty$,
\begin{align}
\|u\|_{L^p_t(\R,L^q_x)} \lesssim \|u_0\|_{L^2_x} + \|F\|_{L^{a'}_t(\R,L^{b'}_x)}, \label{strichartz estimate biharmonic}
\end{align}
and
\begin{align}
\||\nabla|^\gamma u\|_{L^p_t(\R,L^q_x)} \lesssim \||\nabla|^\gamma u_0\|_{L^2_x} + \||\nabla|^{\gamma-1} F\|_{L^2_t(\R,L^{\frac{2d}{d+2}}_x)}. \label{strichartz estimate biharmonic 4order}
\end{align}
\end{coro}
\subsection{Littlewood-Paley decomposition} \label{subsection littlewood-paley decomposition}
Let $\varphi$ be a radial smooth bump function supported in the ball $|\xi|\leq 2$ and  equal to 1 on the ball $|\xi|\leq 1$.  For $M=2^k, k \in \Z$, we define the Littlewood-Paley operators 
\begin{align*}
\widehat{P_{\leq M} f}(\xi) &:= \varphi(M^{-1}\xi) \hat{f}(\xi), \\
\widehat{P_{>M} f}(\xi) &:= (1-\varphi(M^{-1}\xi)) \hat{f}(\xi), \\
\widehat{P_M f}(\xi) &:= (\varphi(M^{-1} \xi) - \varphi(2M^{-1}\xi)) \hat{f}(\xi),
\end{align*}
where $\hat{\cdot}$ is the spatial Fourier transform. Similarly, we can define for $M, M_1 \leq  M_2 \in 2^\Z$,
\[
P_{<M} := P_{\leq M}-P_M, \quad P_{\geq M} := P_{>M}+ P_M, \quad P_{M_1 < \cdot \leq M_2}:= P_{\leq M_2} - P_{\leq M_1} = \sum_{M_1 < M \leq M_2} P_M.
\] 
We recall the following standard Bernstein inequalities (see e.g. \cite[Chapter 2]{BCDfourier} or \cite[Appendix]{Tao}).
\begin{lem}[Bernstein inequalities] \label{lem bernstein inequalities}
Let $\gamma\geq 0$ and $1 \leq p \leq q \leq \infty$. We have
\begin{align*}
\|P_{\geq M} f\|_{L^p_x} &\lesssim M^{-\gamma} \||\nabla|^\gamma P_{\geq M} f\|_{L^p_x}, \\
\|P_{\leq M} |\nabla|^\gamma f\|_{L^p_x} &\lesssim M^\gamma \|P_{\leq M} f\|_{L^p_x}, \\
\|P_M |\nabla|^{\pm \gamma} f\|_{L^p_x} & \sim M^{\pm \gamma} \|P_M f\|_{L^p_x}, \\
\|P_{\leq M} f\|_{L^q_x} &\lesssim M^{\frac{d}{p}-\frac{d}{q}} \|P_{\leq M} f\|_{L^p_x}, \\
\|P_M f\|_{L^q_x} &\lesssim M^{\frac{d}{p}-\frac{d}{q}} \|P_M f\|_{L^p_x}.
\end{align*}
\end{lem}
\subsection{$I$-operator} \label{subsection I-operator}
Let $0\leq \gamma <2$ and $N\gg 1$. We define the Fourier multiplier $I_N$ by
\[
\widehat{I_N f}(\xi):= m_N(\xi) \hat{f}(\xi),
\]
where $m_N$ is a smooth, radially symmetric, non-increasing function such that 
\begin{align*}
m_N(\xi) := \left\{ 
\begin{array}{cl}
1 &\text{if } |\xi|\leq N, \\
(N^{-1}|\xi|)^{\gamma-2} & \text{if } |\xi| \geq 2N.
\end{array}
\right. 
\end{align*}
We shall drop the $N$ from the notation and write $I$ and $m$ instead of $I_N$ and $m_N$. We collect some basic properties of the $I$-operator in the following lemma. 
\begin{lem} [\cite{Dinhmass}] \label{lem properties I operator}
Let $0\leq \sigma \leq \gamma<2$ and $1<q<\infty$. Then
\begin{align}
\|I f\|_{L^q_x} &\lesssim \|f\|_{L^q_x}, \label{property 1} \\
\| |\nabla|^\sigma P_{>N} f\|_{L^q_x} &\lesssim N^{\sigma-2} \|\Delta I f\|_{L^q_x}, \label{property 2} \\
\|\scal{\nabla}^\sigma f\|_{L^q_x} &\lesssim \|\scal{\Delta} I f\|_{L^q_x}, \label{property 3} \\
\|f\|_{H^\gamma_x} \lesssim \|If\|_{H^2_x} &\lesssim N^{2-\gamma} \|f\|_{H^\gamma_x}, \label{property 4} \\
\|If\|_{\dot{H}^2_x} &\lesssim N^{2-\gamma} \|f\|_{\dot{H}^\gamma_x}. \label{property 5} 
\end{align}
\end{lem}
We refer to \cite[Lemma 2.7]{Dinhmass} for the proof of these estimates. We also recall the following product rule which is a modified version of the one given in \cite[Lemma 2.5]{VisanZhang} in the context of nonlinear Schr\"odinger equation. 
\begin{lem} [\cite{Dinhmass}] \label{lem product rule}
Let $\gamma>1, 0<\delta <\gamma-1$ and $1<q, q_1, q_2 <\infty$ be such that $\frac{1}{q}=\frac{1}{q_1}+\frac{1}{q_2}$. Then
\begin{align}
\|I(fg)-(If)g\|_{L^q_x} \lesssim N^{-(2-\gamma+\delta)} \|If\|_{L^{q_1}_x} \|\scal{\nabla}^{2-\gamma+\delta} g\|_{L^{q_2}_x}.
\end{align}
\end{lem}  
We again refer the reader to \cite[Lemma 2.8]{Dinhmass} for the proof of this lemma. A direct consequence of Lemma $\ref{lem product rule}$ and $(\ref{chain rule})$ is the following corollary.
\begin{coro} \label{coro product rule}
Let $\gamma>1, 0<\delta<\gamma-1$ and $1<q, q_1, q_2<\infty$ be such that $\frac{1}{q}=\frac{1}{q_1}+\frac{1}{q_2}$. Then 
\begin{align}
\|\nabla I F(u)-(I\nabla u)F'(u)\|_{L^q_x} \lesssim N^{-(2-\gamma+\delta)} \|\nabla I u\|_{L^{q_1}_x} \|\scal{\nabla}^{2-\gamma+\delta} F'(u)\|_{L^{q_2}_x}. \label{product rule}
\end{align}
\end{coro}
\subsection{Interaction Morawetz inequality} \label{subsection interaction morawetz inequality}
We now recall the interaction Morawetz inequality for the nonlinear fourth-order Schr\"odinger equation. 
\begin{prop}[Interaction Morawetz inequality \cite{Pausadercubic}, \cite{MiaoWuZhang}]
Let $d\geq 5$, $J$ be a compact time interval and $u$ a solution to \emph{(NL4S)} on the spacetime slab $J\times \R^d$. Then we have the following a priori estimate:
\begin{align}
\||\nabla|^{-\frac{d-5}{4}} u\|_{L^4_t(J, L^4_x)} \lesssim \|u_0\|^{\frac{1}{2}}_{L^2_x} \|u\|^{\frac{1}{2}}_{L^\infty_t(J, \dot{H}^{\frac{1}{2}}_x)}. \label{interaction morawetz}
\end{align} 
\end{prop}
This estimate was first established by Pausader in \cite{Pausadercubic} for $d\geq 7$. Later, Miao-Wu-Zhang in \cite{MiaoWuZhang} extended this interaction Morawetz estimate to $d\geq 5$. By interpolating $(\ref{interaction morawetz})$ and the trivial estimate for $0<\sigma\leq \gamma$,
\[
\|u\|_{L^\infty_t(J, \dot{H}^\sigma_x)} \leq \|u\|_{L^\infty_t(J, \dot{H}^\sigma_x)},
\]
we obtain
\begin{align}
\|u\|_{M^\sigma(J)} \lesssim \Big(\|u_0\|_{L^2_x} \|u\|_{L^\infty_t(J,\dot{H}^{\frac{1}{2}}_x)}\Big)^{\frac{2\sigma}{d-5+4\sigma}} \|u\|_{L^\infty_t(J,\dot{H}^\sigma_x)}^{\frac{d-5}{d-5+4\sigma}}, \label{interaction morawetz interpolation}
\end{align}
where
\begin{align}
\|u\|_{M^\sigma(J)}:= \|u\|_{L^{\frac{d-5+4\sigma}{\sigma}}_t L^{\frac{2(d-5+4\sigma)}{d-5+2\sigma}}_x}. \label{define M sigma}
\end{align} 
\section{Almost conservation law} \label{section almost conservation law}
\setcounter{equation}{0}
For any spacetime slab $J\times \R^d$, we define
\[
Z_I(J):= \sup_{(p,q)\in B} \|\scal{\Delta} I u\|_{L^p_t(J, L^q_x)}.
\]
Note that in our considerations, the biharmonic admissible condition $(p,q)\in B$ ensures $q<\infty$. Let us start with the following commutator estimates.
\begin{lem} \label{lem commutator estimate}
Let $5\leq d \leq 11, \frac{2+\Gace}{2}<\gamma<2, 0<\delta< \min \{2\gamma-\Gace-2, \gamma-1\}, 0<\sigma \leq \gamma$ and 
\begin{align*}
\max\left\{\frac{8(d-5+4\sigma)}{d(d-5+2\sigma)+8\sigma}, 1 \right\} <\nu-1< \min\left\{\frac{d-5+4\sigma}{2\sigma},\frac{8}{d-2\gamma}\right\}. \nonumber
\end{align*} 
Assume that 
\[
\|u\|_{M^\sigma(J)} \leq \mu,
\]
for some small constant $\mu>0$. Then
\begin{align}
\|\nabla I F(u)-(I\nabla u) F'(u)\|_{L^2_t(J, L^{\frac{2d}{d+2}}_x)} &\lesssim N^{-(2-\gamma+\delta)} Z_I(J)\left(\mu^\theta Z_I^{1-\theta}(J)+ Z_I(J)\right)^{\nu-1} \label{commutator estimate 1} \\
\|\nabla I F(u)\|_{L^2_t(J, L^{\frac{2d}{d+2}}_x)} &\lesssim N^{-(2-\gamma+\delta)}Z_I^{\nu}(J) + \mu^{(\nu-1)\theta}Z_I^{1+(\nu-1)(1-\theta)}(J), \label{commutator estimate 2}
\end{align}
where
\begin{align}
\theta:= \frac{(d-5+4\sigma)(8-(d-4)(\nu-1))}{2(\nu-1)(2(d-5)+(12-d)\sigma)} \in (0,1). \label{define theta}
\end{align}  
\end{lem}
\begin{proof}
For simplifying the presentation, we shall drop the dependence on the time interval $J$. Denote 
\[
\varepsilon:= \frac{4(\nu-1)\sigma}{d-5+4\sigma -2(\nu-1)\sigma}.
\] 
It is easy to see from our assumptions that $\varepsilon>0$. We next apply $(\ref{product rule})$ with $q=\frac{2d}{d+2}, q_1=\frac{2d(2+\varepsilon)}{(d-2)(2+\varepsilon)-8}$ and $q_2=\frac{d(2+\varepsilon)}{2\varepsilon+8}$ to get
\[
\|\nabla I F(u)-(I\nabla u) F'(u)\|_{L^{\frac{2d}{d+2}}_x} \lesssim N^{-\alpha} \|\nabla I u\|_{L^{\frac{2d(2+\varepsilon)}{(d-2)(2+\varepsilon)-8}}_x} \|\scal{\nabla}^\alpha F'(u)\|_{L^{\frac{d(2+\varepsilon)}{2\varepsilon+8}}_x},
\]
where $\alpha=2-\gamma+\delta$. Note that $q_1$ is well-defined since $(d-2)(2+\varepsilon)-8>0$. We then apply H\"older's inequality in time to have
\begin{align}
\|\nabla I F(u)-(I\nabla u) F'(u)\|_{L^2_tL^{\frac{2d}{d+2}}_x} \lesssim N^{-\alpha} \|\nabla I u\|_{L^{2+\varepsilon}_tL^{\frac{2d(2+\varepsilon)}{(d-2)(2+\varepsilon)-8}}_x} \|\scal{\nabla}^\alpha F'(u)\|_{L^{\frac{2(2+\varepsilon)}{\varepsilon}}_tL^{\frac{d(2+\varepsilon)}{2\varepsilon+8}}_x}. \label{commutator estimate 1 sub 1}
\end{align}
For the first factor in the right hand side of $(\ref{commutator estimate 1 sub 1})$, we use the Sobolev embedding to obtain
\begin{align}
\|\nabla I u\|_{L^{2+\varepsilon}_t L^{\frac{2d(2+\varepsilon)}{(d-2)(2+\varepsilon)-8}}_x} \lesssim \|\Delta I u\|_{L^{2+\varepsilon}_t L^{\frac{2d(2+\varepsilon)}{d(2+\varepsilon)-8}}_x} \lesssim Z_I, \label{first factor estimate}
\end{align}
where $\Big(2+\varepsilon, \frac{2d(2+\varepsilon)}{d(2+\varepsilon)-8}\Big)$ is a biharmonic admissible pair. To treat the second factor in the right hand side of $(\ref{commutator estimate 1 sub 1})$, we note that $\alpha < \gamma-\Gact$ by our assumption on $\delta$. Thus
\begin{align}
\|\scal{\nabla}^\alpha F'(u)\|_{L^{\frac{2(2+\varepsilon)}{\varepsilon}}_tL^{\frac{d(2+\varepsilon)}{2\varepsilon+8}}_x} &\lesssim \|\scal{\nabla}^{\gamma-\Gact} F'(u)\|_{L^{\frac{2(2+\varepsilon)}{\varepsilon}}_tL^{\frac{d(2+\varepsilon)}{2\varepsilon+8}}_x} \nonumber \\
&\lesssim  \|F'(u)\|_{L^{\frac{2(2+\varepsilon)}{\varepsilon}}_tL^{\frac{d(2+\varepsilon)}{2\varepsilon+8}}_x} + \||\nabla|^{\gamma-\Gact} F'(u)\|_{L^{\frac{2(2+\varepsilon)}{\varepsilon}}_tL^{\frac{d(2+\varepsilon)}{2\varepsilon+8}}_x}. \label{commutator estimate}
\end{align}
Since $F'(u)=O(|u|^{\nu-1})$, we bound the first term in $(\ref{commutator estimate})$ as  
\[
\|F'(u)\|_{L^{\frac{2(2+\varepsilon)}{\varepsilon}}_tL^{\frac{d(2+\varepsilon)}{2\varepsilon+8}}_x} \lesssim \|u\|^{\nu-1}_{L^{\frac{2(\nu-1)(2+\varepsilon)}{\varepsilon}}_tL^{\frac{d(\nu-1)(2+\varepsilon)}{2\varepsilon+8}}_x}. 
\]
By the choice of $\varepsilon$, we have
\[
\frac{2(\nu-1)(2+\varepsilon)}{\varepsilon}=\frac{d-5+4\sigma}{\sigma}, \quad \frac{d(\nu-1)(2+\varepsilon)}{2\varepsilon+8}=\frac{d(\nu-1)(d-5+4\sigma)}{4(d-5+4\sigma-(\nu-1)\sigma)}.
\]
We next split $u:= P_{\leq N}u + P_{>N}u$. For the low frequency part, we estimate
\begin{align}
\|P_{\leq N} u\|_{L^{\frac{d-5+4\sigma}{\sigma}}_tL^{\frac{d(\nu-1)(d-5+4\sigma)}{4(d-5+4\sigma-(\nu-1)\sigma)}}_x} &\lesssim \|P_{\leq N} u\|^\theta_{M^\sigma} \|P_{\leq N}u\|^{1-\theta}_{L^{\frac{d-5+4\sigma}{\sigma}}_t L^{\frac{2d(d-5+4\sigma)}{(d-4)(d-5+4\sigma)-8\sigma}}_x} \nonumber \\
&\lesssim  \|P_{\leq N} u\|^\theta_{M^\sigma} \|\Delta P_{\leq N}u\|^{1-\theta}_{L^{\frac{d-5+4\sigma}{\sigma}}_t L^{\frac{2d(d-5+4\sigma)}{d(d-5+4\sigma)-8\sigma}}_x} \nonumber \\
&\lesssim \mu^\theta Z^{1-\theta}_I, \label{low frequency estimate}
\end{align}
where $\theta$ is given in $(\ref{define theta})$. Here the first line follows from H\"older's inequality, and the second line makes use of the Sobolev embedding. The last inequality uses the fact that $\Big(\frac{d-5+4\sigma}{\sigma}, \frac{2d(d-5+4\sigma)}{d(d-5+4\sigma)-8\sigma}\Big)$ is biharmonic admissible. Note that our assumptions ensure $\theta \in (0,1)$. For the high frequency part, the Sobolev embedding gives
\begin{align}
\|P_{>N} u\|_{L^{\frac{d-5+4\sigma}{\sigma}}_t L^{\frac{d(\nu-1)(d-5+4\sigma)}{4(d-5+4\sigma-(\nu-1)\sigma)}}_x} \lesssim \| |\nabla|^{\Gact} P_{>N} u \|_{L^{\frac{d-5+4\sigma}{\sigma}}_t L^{\frac{2d(d-5+4\sigma)}{d(d-5+4\sigma)-8\sigma}}_x } \lesssim N^{\Gact-2} Z_I. \label{high frequency estimate}
\end{align}
Here $\Big(\frac{d-5+4\sigma}{\sigma}, \frac{2d(d-5+4\sigma)}{d(d-5+4\sigma)-8\sigma} \Big)$ is biharmonic admissible. Thus, we obtain
\begin{align}
\|u\|_{L^{\frac{2(\nu-1)(2+\varepsilon)}{\varepsilon}}_t L^{\frac{d(\nu-1)(2+\varepsilon)}{2\varepsilon+8}}_x} \lesssim \mu^\theta Z_I^{1-\theta} + N^{\Gact-2}Z_I. \label{u norm estimate}
\end{align}
In particular,
\begin{align}
\|F'(u)\|_{L^{\frac{2(2+\varepsilon)}{\varepsilon}}_tL^{\frac{d(2+\varepsilon)}{2\varepsilon+8}}_x} \lesssim (\mu^\theta Z^{1-\theta}_I+ Z_I)^{\nu-1}. \label{nonhomogeneous term estimate}
\end{align}
We next treat the second term in $(\ref{commutator estimate})$. Since $\nu-1>1$, we are able to apply Lemma $\ref{lem fractional chain rule C1}$ to get
\begin{align}
\||\nabla|^{\gamma-\Gact} F'(u)\|_{L^{\frac{2(2+\varepsilon)}{\varepsilon}}_t L^{\frac{d(2+\varepsilon)}{2\varepsilon+8}}_x } &\lesssim \|F''(u)\|_{L^{\frac{2(\nu-1)(2+\varepsilon)}{\varepsilon(\nu-2)}}_t L^{\frac{d(\nu-1)(2+\varepsilon)}{(2\varepsilon+8)(\nu-2)}}_x} \||\nabla|^{\gamma-\Gact} u\|_{L^{\frac{2(\nu-1)(2+\varepsilon)}{\varepsilon}}_t L^{\frac{d(\nu-1)(2+\varepsilon)}{2\varepsilon+8}}_x} \nonumber \\
&\lesssim \|u\|^{\nu-2}_{L^{\frac{2(\nu-1)(2+\varepsilon)}{\varepsilon}}_t L^{\frac{d(\nu-1)(2+\varepsilon)}{2\varepsilon+8}}_x} \||\nabla|^{\gamma-\Gact} u\|_{L^{\frac{2(\nu-1)(2+\varepsilon)}{\varepsilon}}_t L^{\frac{d(\nu-1)(2+\varepsilon)}{2\varepsilon+8}}_x}, \label{homogeneous term}
\end{align}
where $F''(u)=O(|u|^{\nu-2})$. The first factor in the right hand side of $(\ref{homogeneous term})$ is treated in $(\ref{u norm estimate})$. For the second factor, we split $u:= P_{\leq 1}u + P_{1<\cdot \leq N} u + P_{>N} u$. We use Bernstein inequality and estimate as in $(\ref{low frequency estimate})$,
\begin{align*}
\||\nabla|^{\gamma-\Gact} P_{\leq 1} u\|_{L^{\frac{2(\nu-1)(2+\varepsilon)}{\varepsilon}}_t L^{\frac{d(\nu-1)(2+\varepsilon)}{2\varepsilon+8}}_x} \lesssim \|P_{\leq 1} u\|_{L^{\frac{2(\nu-1)(2+\varepsilon)}{\varepsilon}}_t L^{\frac{d(\nu-1)(2+\varepsilon)}{2\varepsilon+8}}_x} \lesssim \mu^\theta Z_I^{1-\theta}.
\end{align*}
The intermediate term is bounded by
\begin{align*}
\||\nabla|^{\gamma-\Gact} P_{1<\cdot \leq N} u\|_{L^{\frac{2(\nu-1)(2+\varepsilon)}{\varepsilon}}_t L^{\frac{d(\nu-1)(2+\varepsilon)}{2\varepsilon+8}}_x} &\lesssim \||\nabla|^{\gamma} P_{1<\cdot\leq N} u\|_{L^{\frac{2(\nu-1)(2+\varepsilon)}{\varepsilon}}_t L^{\frac{2d(\nu-1)(2+\varepsilon)}{d(\nu-1)(2+\varepsilon)-4\varepsilon}}_x} \\
&\lesssim  \|\Delta I P_{1<\cdot\leq N} u\|_{L^{\frac{2(\nu-1)(2+\varepsilon)}{\varepsilon}}_t L^{\frac{2d(\nu-1)(2+\varepsilon)}{d(\nu-1)(2+\varepsilon)-4\varepsilon}}_x}  \\
&\lesssim  Z_I. 
\end{align*}
Here we use
\[
\||\nabla|^{\gamma} (\Delta I)^{-1}\|_{L^{\frac{2d(\nu-1)(2+\varepsilon)}{d(\nu-1)(2+\varepsilon)-4\varepsilon}}_x \rightarrow L^{\frac{2d(\nu-1)(2+\varepsilon)}{d(\nu-1)(2+\varepsilon)-4\varepsilon}}_x} \lesssim 1,
\]
and the fact $\Big(\frac{2(\nu-1)(2+\varepsilon)}{\varepsilon}, \frac{2d(\nu-1)(2+\varepsilon)}{d(\nu-1)(2+\varepsilon)-4\varepsilon}\Big)$ is biharmonic admissible. 
Finally, we use $(\ref{property 2})$ to estimate
\begin{align*}
\||\nabla|^{\gamma-\Gact} P_{>N} u\|_{L^{\frac{2(\nu-1)(2+\varepsilon)}{\varepsilon}}_t L^{\frac{d(\nu-1)(2+\varepsilon)}{2\varepsilon+8}}_x} &\lesssim \||\nabla|^{\gamma} P_{>N} u\|_{L^{\frac{2(\nu-1)(2+\varepsilon)}{\varepsilon}}_t L^{\frac{2d(\nu-1)(2+\varepsilon)}{d(\nu-1)(2+\varepsilon)-4\varepsilon}}_x} \\
&\lesssim N^{\gamma -2} \|\Delta P_{>N} u\|_{L^{\frac{2(\nu-1)(2+\varepsilon)}{\varepsilon}}_t L^{\frac{2d(\nu-1)(2+\varepsilon)}{d(\nu-1)(2+\varepsilon)-4\varepsilon}}_x} \\
&\lesssim N^{\gamma-2} Z_I.
\end{align*}
Combining three terms yields
\begin{align}
\||\nabla|^{\gamma-\Gact} u\|_{L^{\frac{2(2+\varepsilon)}{\varepsilon}}_t L^{\frac{d(\nu-1)(2+\varepsilon)}{2\varepsilon+8}}_x } \lesssim \mu^\theta Z_I^{1-\theta} +Z_I. \label{homogeneous term estimate}
\end{align}
Collecting $(\ref{commutator estimate 1 sub 1}), (\ref{first factor estimate}), (\ref{nonhomogeneous term estimate}), (\ref{homogeneous term})$ and $(\ref{homogeneous term estimate})$, we show the first estimate $(\ref{commutator estimate 1})$. \newline
\indent We now prove $(\ref{commutator estimate 2})$. By triangle inequality,
\[
\|I\nabla F(u)\|_{L^2_t L^{\frac{2d}{d+2}}_x} \leq \|(I\nabla u)F'(u)\|_{L^2_t L^{\frac{2d}{d+2}}_x} + \|I\nabla F(u)-(I\nabla u)F'(u)\|_{L^2_t L^{\frac{2d}{d+2}}_x}.
\]
We have from H\"older's inequality, $(\ref{first factor estimate})$ and $(\ref{u norm estimate})$ that
\begin{align}
\|(I\nabla u) F'(u)\|_{L^2_t L^{\frac{2d}{d+2}}_x} &\lesssim \|I \nabla u\|_{L^{2+\varepsilon}_t L^{\frac{2d(2+\varepsilon)}{(d-2)(2+\varepsilon)-8}}_x} \|F'(u)\|_{L^{\frac{2(2+\varepsilon)}{\varepsilon}}_tL^{\frac{d(2+\varepsilon)}{2\varepsilon+8}}_x} \nonumber \\
&\lesssim \|\Delta I u\|_{L^{2+\varepsilon}_t L^{\frac{2d(2+\varepsilon)}{d(2+\varepsilon)-8}}_x} \|u\|^{\nu-1}_{L^{\frac{2(\nu-1)(2+\varepsilon)}{\varepsilon}}_t L^{\frac{d(\nu-1)(2+\varepsilon)}{2\varepsilon+8}}_x} \nonumber \\
&\lesssim Z_I (\mu^\theta Z_I^{1-\theta}+ N^{\Gact-2} Z_I)^{\nu-1} \nonumber \\
&\lesssim \mu^{\theta(\nu-1)} Z_I^{1+(1-\theta)(\nu-1)} + N^{(\Gact-2)(\nu-1)} Z_I^\nu. \label{first term triangle estimate}
\end{align}
The estimate $(\ref{commutator estimate 2})$ follows easily from $(\ref{commutator estimate 1})$ and $(\ref{first term triangle estimate})$. Note that by our assumptions, $\alpha=2-\gamma+\delta<\gamma-\Gact<2-\Gact$. The proof is complete.
\end{proof}
\begin{rem} \label{rem commutator estimate}
The estimates $(\ref{commutator estimate 1})$ and $(\ref{commutator estimate 2})$ still hold for $\nu-1=\frac{d-5+4\sigma}{2\sigma}$. Indeed, the proof of Lemma $\ref{lem commutator estimate}$ is valid for $\varepsilon=\infty$. 
\end{rem}
We are now able to prove the almost conservation law for the modified energy functional $E(Iu)$, where
\[
E(Iu(t))= \frac{1}{2}\|Iu(t)\|^2_{\dot{H}^2_x} + \frac{1}{\nu+1}\|Iu(t)\|^{\nu+1}_{L^{\nu+1}_x}.
\]
\begin{prop}\label{prop almost conservation law}
Let $5\leq d \leq 11$,
\[ 
\max \left\{\frac{3}{2}+\frac{\Gace}{4}, 4-\nu, \frac{2}{\nu-1}+\frac{(\nu-2)\Gace}{\nu-1} \right\}<\gamma<2,
\] 
$0<\delta<\min \{2\gamma-3-\frac{\Gace}{2}, \gamma+\nu-4, (\nu-1)\gamma-2-(\nu-2)\Gace\}, 0<\sigma \leq \gamma$ and 
\begin{align}
\max\left\{\frac{8(d-5+4\sigma)}{d(d-5+2\sigma)+8\sigma}, 
1 \right\} < \nu-1<\min\left\{\frac{d-5+4\sigma}{2\sigma},\frac{8}{d-2\gamma}\right\}. \nonumber
\end{align}  
Assume that $u \in L^\infty([0,T],\mathscr{S}(\R^d))$ is a solution to \emph{(NL4S)} on a time interval $[0,T]$, and satisfies $\|I u_0\|_{H^2_x} \leq 1$. Assume in addition that $u$ satisfies the a priori bound 
\[
\|u\|_{M^\sigma([0,T])} \leq \mu,
\]
for some small constant $\mu>0$. Then, for $N$ sufficiently large,
\begin{align}
\sup_{t\in[0,T]} |E(Iu(t))-E(Iu_0)| \lesssim N^{-(2-\gamma+\delta)}. \label{almost conservation law}
\end{align}
Here the implicit constant depends only on the size of $E(Iu_0)$. 
\end{prop}
\begin{rem} \label{rem almost conservation law}
As in Remark $\ref{rem commutator estimate}$, the estimate $(\ref{almost conservation law})$ is still valid for $\nu-1=\frac{d-5+4\sigma}{2\sigma}$.
\end{rem}
\noindent \textit{Proof of \emph{Proposition} $\ref{prop almost conservation law}$.} We firstly note that our assumptions on $\gamma$ and $\delta$ satisfy the assumptions given in Lemma $\ref{lem commutator estimate}$. It allows us to use the estimates given in Lemma $\ref{lem commutator estimate}$. \newline
\indent We begin by controlling the size of $Z_I$. By applying $I, \Delta I$ to (NL4S), and using Strichartz estimates $(\ref{strichartz estimate biharmonic}), (\ref{strichartz estimate biharmonic 4order})$, we get
\begin{align}
Z_I \lesssim \|I u_0\|_{H^2_x} + \|IF(u)\|_{L^2_t L^{\frac{2d}{d+4}}_x}
 + \|\nabla I F(u)\|_{L^2_t L^{\frac{2d}{d+2}}_x}. \label{local estimate 1}
\end{align}
Using $(\ref{commutator estimate 2})$, we have
\begin{align}
\|\nabla I F(u)\|_{L^2_t L^{\frac{2d}{d+2}}_x} \lesssim N^{-(2-\gamma+\delta)}Z_I^{\nu} + \mu^{(\nu-1)\theta}Z_I^{1+(\nu-1)(1-\theta)}. \label{local estimate 2}
\end{align}
Next, we drop the $I$-operator and use H\"older's inequality together with $(\ref{u norm estimate})$ to estimate
\begin{align*}
\|IF(u)\|_{L^2_t L^{\frac{2d}{d+4}}_x} &\lesssim \||u|^{\nu-1}\|_{L^{\frac{2(2+\varepsilon)}{\varepsilon}}_t L^{\frac{d(2+\varepsilon)}{2\varepsilon+8}}_x} \|u\|_{L^{2+\varepsilon}_t L^{\frac{2d(2+\varepsilon)}{d(2+\varepsilon)-8}}_x} \\
&\lesssim \|u\|^{\nu-1}_{L^{\frac{2(\nu-1)(2+\varepsilon)}{\varepsilon}}_t L^{\frac{d(\nu-1)(2+\varepsilon)}{2\varepsilon+8}}_x} \|u\|_{L^{2+\varepsilon}_t L^{\frac{2d(2+\varepsilon)}{d(2+\varepsilon)-8}}_x} \\
&\lesssim Z_I(\mu^\theta Z_I^{1-\theta} + N^{\Gact-2} Z_I)^{\nu-1} \\
&\lesssim \mu^{(\nu-1)\theta} Z_I^{1+(\nu-1)(1-\theta)} + N^{(\Gact-2)(\nu-1)} Z_I^\nu.
\end{align*}
Here $\Big(2+\varepsilon, \frac{2d(2+\varepsilon)}{d(2+\varepsilon)-8} \Big)$ is biharmonic admissible. We thus get
\[
Z_I \lesssim \|Iu_0\|_{H^2_x} + N^{-(2-\gamma+\delta)}Z_I^{\nu} + \mu^{(\nu-1)\theta}Z_I^{1+(\nu-1)(1-\theta)}.
\]
By taking $\mu$ sufficiently small and $N$ sufficiently large and using the assumption $\|Iu_0\|_{H^2_x} \leq 1$, the continuity argument gives
\begin{align}
Z_I \lesssim \|Iu_0\|_{H^2_x}\leq 1. \label{control Z_I norm}
\end{align}
\indent Now, let $F(u)=|u|^{\nu-1}u$. A direct computation shows
\[
\partial_t E(Iu(t)) = \re{ \int \overline{I\partial_t u} (\Delta^2 Iu + F(Iu))  dx}.
\]
By the Fundamental Theorem of Calculus, 
\[
E(Iu(t))-E(Iu_0) = \int_0^t \partial_s E(Iu(s)) ds  = \re{ \int_0^t \int \overline{I \partial_s u} (\Delta^2 Iu + F(Iu)) dx ds}.  
\]
Using $I\partial_t u = i\Delta^2 Iu + i IF(u)$, we see that
\begin{align}
E(Iu(t))-E(Iu_0) &= \re{ \int_0^t \int \overline{I \partial_s u} (F(Iu)-IF(u)) dx ds} \nonumber \\
&= \im{ \int_0^t \int \overline{\Delta^2 Iu + IF(u)} (F(Iu)-IF(u)) dx ds} \nonumber \\
&= \im{ \int_0^t \int \overline{\Delta Iu} \Delta(F(Iu)-IF(u)) dx ds } \nonumber\\
&\mathrel{\phantom{=}} + \im{\int_0^t \int \overline{IF(u)} (F(Iu)-IF(u)) dx ds} \nonumber 
\end{align}
We next write 
\begin{align*}
\Delta(F(Iu)- IF(u)) &= (\Delta I u) F'(Iu) + |\nabla I u|^2F''(Iu) - I(\Delta F'(u)) - I(|\nabla u|^2 F''(u)) \\
&= (\Delta I u) (F'(Iu)-F'(u)) + |\nabla I u|^2(F''(Iu)-F''(u))  + \nabla Iu \cdot (\nabla I u - \nabla u) F''(u)  \\
& \mathrel{\phantom{=}} + (\Delta I u) F'(u) -I(\Delta u F'(u)) + (I\nabla u)\cdot \nabla u F''(u) - I(\nabla u \cdot \nabla u F''(u)).
\end{align*}
Therefore,
\begin{align}
E(Iu(t))-E(Iu_0) &= \im{ \int_0^t \int \overline{\Delta Iu} \Delta I u (F'(Iu)-F'(u)) dx ds} \label{almost conservation estimate 1} \\
& \mathrel{\phantom{=}} + \im{ \int_0^t \int \overline{\Delta Iu} |\nabla I u|^2(F''(Iu)-F''(u)) dx ds} \label{almost conservation estimate 2} \\
& \mathrel{\phantom{=}} + \im{ \int_0^t \int \overline{\Delta Iu} \nabla Iu \cdot (\nabla I u - \nabla u) F''(u) dx ds} \label{almost conservation estimate 3} \\
& \mathrel{\phantom{=}} + \im{ \int_0^t \int \overline{\Delta Iu} [(\Delta I u) F'(u) -I(\Delta u F'(u))] dx ds} \label{almost conservation estimate 4} \\
& \mathrel{\phantom{=}} + \im{ \int_0^t \int \overline{\Delta Iu} [(I\nabla u)\cdot \nabla u F''(u) - I(\nabla u \cdot \nabla u F''(u))]dx ds} \label{almost conservation estimate 5} \\
&\mathrel{\phantom{=}} + \im{\int_0^t \int \overline{IF(u)} (F(Iu)-IF(u)) dx ds} \label{almost conservation estimate 6}.
\end{align}
Let us consider $(\ref{almost conservation estimate 1})$. By H\"older's inequality, we estimate
\begin{align}
|(\ref{almost conservation estimate 1})| &\lesssim \|\Delta I u\|_{L^2_t L^{\frac{2d}{d-4}}_x} \|\Delta I u\|_{L^{2+\varepsilon}_t L^{\frac{2d(2+\varepsilon)}{d(2+\varepsilon) -8}}_x} \|F'(Iu)-F'(u)\|_{L^{\frac{2(2+\varepsilon)}{\varepsilon}}_t L^{\frac{d(2+\varepsilon)}{2\varepsilon+8}}_x} \nonumber \\
& \lesssim Z_I^2 \| |Iu-u|(|Iu|+|u|)^{\nu-2} \|_{L^{\frac{2(2+\varepsilon)}{\varepsilon}}_t L^{\frac{d(2+\varepsilon)}{2\varepsilon+8}}_x} \nonumber \\
& \lesssim Z_I^2 \|P_{>N} u\|_{L^{\frac{2(\nu-1)(2+\varepsilon)}{\varepsilon}}_t L^{\frac{d(\nu-1)(2+\varepsilon)}{2\varepsilon+8}}_x } \|u\|^{\nu-2}_{L^{\frac{2(\nu-1)(2+\varepsilon)}{\varepsilon}}_t L^{\frac{d(\nu-1)(2+\varepsilon)}{2\varepsilon+8}}_x}. \label{almost conservation estimate 1 sub 1}
\end{align}
Combining $(\ref{almost conservation estimate 1 sub 1}), (\ref{high frequency estimate})$ and $(\ref{u norm estimate})$, we get
\begin{align}
|(\ref{almost conservation estimate 1})| \lesssim N^{\Gact-2} Z_I^3(\mu^\theta Z_I^{1-\theta} + Z_I)^{\nu-2}. \label{almost conservation estimate 1 final}
\end{align}
In order to treat $(\ref{almost conservation estimate 2})$, we need to separate two cases $0<\nu-2<1$ and $1\leq \nu-2$. \newline
\indent If $0<\nu-2<1$, then using $F''(z)=O(|z|^{\nu-2})$, we have
\[
|F''(z)-F''(\zeta)| \lesssim |z-\zeta|^{\nu-2}, \quad \forall z,\zeta \in \C. 
\]
Moreover, there exists $k\gg 1$ so that $k(\nu-2)\geq 2$. 
By H\"older's inequality, 
\begin{align}
|(\ref{almost conservation estimate 2})| &\lesssim \|\Delta I u\|_{L^2_t L^{\frac{2d}{d-4}}_x} \|\nabla I u\|^2_{L^{\frac{4k}{k-2}}_t L^{q}_x} \|F''(Iu)-F''(u)\|_{L^{k}_t L^{r}_x} \nonumber \\
& \lesssim \|\Delta I u\|_{L^2_t L^{\frac{2d}{d-4}}_x} \||\nabla|^{1+\frac{\Gact}{2}} I u\|^2_{L^{\frac{4k}{k-2}}_t L^{q^\star}_x} \||Iu-u|^{\nu-2} \|_{L^{k}_t L^{r}_x} \nonumber \\
& \lesssim Z_I^3 \|P_{>N} u\|^{\nu-2}_{L^{k(\nu-2)}_t L^{r(\nu-2)}_x } \nonumber \\
&\lesssim Z_I^3 \||\nabla|^{\Gact} P_{>N} u\|^{\nu-2}_{L^{k(\nu-2)}_t L^{r^\star}_x} \nonumber\\
&\lesssim N^{(\Gact-2)(\nu-2)} Z_I^{\nu+1}, \nonumber
\end{align}
where 
\begin{align}
\begin{array}{rlrl}
q&:=\frac{4kd(\nu-1)}{(kd-4(k-2))(\nu-1)+8k}, &q^\star &:=\frac{2kd}{kd-2(k-2)}, \\
r&:=\frac{kd(\nu-1)}{4(k-1)(\nu-1)-4k}, & r^\star &:=\frac{2kd}{kd(\nu-2)-8}. 
\end{array}
\label{define q q_star r r_star}
\end{align}
Here we drop the $I$-operator and use $(\ref{property 3})$ with the fact $1+\frac{\Gact}{2}<\gamma<2$ to have the third line. Note that $\Big(\frac{4k}{k-2}, q^\star \Big)$ and $\Big(k(\nu-2), r^\star \Big)$ are biharmonic admissible. The last line follows from $(\ref{property 2})$.\newline
\indent If $1\leq \nu-2$, then 
\[
|F''(z)-F''(\zeta)| \lesssim |z-\zeta|(|z|+|\zeta|)^{\nu-3}, \quad \forall z, \zeta \in \C.
\]
We estimate
\begin{align}
|(\ref{almost conservation estimate 2})| &\lesssim \|\Delta I u\|_{L^2_t L^{\frac{2d}{d-4}}_x} \|\nabla I u\|^2_{L^{\frac{4k}{k-2}}_t L^{q}_x} \|F''(Iu)-F''(u)\|_{L^{k}_t L^{r}_x} \nonumber \\
& \lesssim \|\Delta I u\|_{L^2_t L^{\frac{2d}{d-4}}_x} \||\nabla|^{1+\frac{\Gact}{2}} I u\|^2_{L^{\frac{4k}{k-2}}_t L^{q^\star}_x} \|F''(Iu)-F''(u)\|_{L^{k}_t L^{r}_x} \nonumber \\
& \lesssim Z_I^3 \|Iu-u\|_{L^{k(\nu-2)}_t L^{r(\nu-2)}_x} \||Iu|+|u|\|^{\nu-3}_{L^{k(\nu-2)}_t L^{r(\nu-2)}_x} \nonumber \\
&\lesssim Z_I^3 \||\nabla|^{\Gact} P_{>N} u\|_{L^{k(\nu-2)}_t L^{r^\star}_x} \||\nabla|^{\Gact}u\|^{\nu-3}_{L^{k(\nu-2)}_t L^{r^\star}_x} \nonumber\\
&\lesssim N^{\Gact-2}Z_I^{\nu+1}. \nonumber
\end{align}
Thus, collecting two cases, we obtain
\begin{align}
|(\ref{almost conservation estimate 2})| \lesssim N^{\min\{\nu-2, 1\}(\Gact-2)}Z_I^{\nu+1}. \label{almost conservation estimate 2 final}
\end{align}
We next estimate
\begin{align}
|(\ref{almost conservation estimate 3})| &\lesssim \|\Delta I u\|_{L^2_t L^{\frac{2d}{d-4}}_x} \|\nabla I u\|_{L^{\frac{4k}{k-2}}_t L^{q}_x} \|\nabla I u- \nabla u\|_{L^{\frac{4k}{k-2}}_t L^{q}_x}  \|F''(u)\|_{L^{k}_t L^{r}_x} \nonumber \\
&\lesssim \|\Delta I u\|_{L^2_t L^{\frac{2d}{d-4}}_x} \|\nabla I u\|_{L^{\frac{4k}{k-2}}_t L^{q}_x} \|\nabla P_{>N} u\|_{L^{\frac{4k}{k-2}}_t L^{q}_x}  \|u\|^{\nu-2}_{L^{k(\nu-2)}_t L^{r(\nu-2)}_x} \nonumber \\
&\lesssim \|\Delta I u\|_{L^2_t L^{\frac{2d}{d-4}}_x} \||\nabla|^{1+\frac{\Gact}{2}} I u\|_{L^{\frac{4k}{k-2}}_t L^{q^\star}_x} \||\nabla|^{1+\frac{\Gact}{2}} P_{>N} u\|_{L^{\frac{4k}{k-2}}_t L^{q^\star}_x}  \||\nabla|^{\Gact}u\|^{\nu-2}_{L^{k(\nu-2)}_t L^{r^\star}_x} \nonumber \\
&\lesssim N^{\frac{\Gact-2}{2}} Z_I^{\nu+1}. \label{almost conservation estimate 3 final}
\end{align}
We next consider the term $(\ref{almost conservation estimate 4})$. Using the notation given in Lemma $\ref{lem commutator estimate}$, we apply Corollary $\ref{coro product rule}$ with $q=\frac{2d}{d+4}, q_1= \frac{2d(2+\varepsilon)}{d(2+\varepsilon)-8}$ and $q_2=\frac{d(2+\varepsilon)}{2\varepsilon+8}$ to have
\[
\|(\Delta Iu) F'(u)-I(\Delta u F'(u))\|_{L^{\frac{2d}{d+4}}_x} \lesssim N^{-\alpha}\|\Delta I u\|_{L^{\frac{2d(2+\varepsilon)}{d(2+\varepsilon)-8}}_x} \|\scal{\nabla}^\alpha F'(u)\|_{L^{\frac{d(2+\varepsilon)}{2\varepsilon+8}}_x},
\]
where $\alpha=2-\gamma+\delta$. By H\"older's inequality,
\[
\|(\Delta Iu) F'(u)-I(\Delta u F'(u))\|_{L^2_t L^{\frac{2d}{d+4}}_x} \lesssim N^{-\alpha}\|\Delta I u\|_{L^{2+\varepsilon}_t L^{\frac{2d(2+\varepsilon)}{d(2+\varepsilon)-8}}_x} \|\scal{\nabla}^\alpha F'(u)\|_{L^{\frac{2(2+\varepsilon)}{\varepsilon}}_t L^{\frac{d(2+\varepsilon)}{2\varepsilon+8}}_x}.
\]
We have from $(\ref{commutator estimate}), (\ref{nonhomogeneous term estimate})$ and $(\ref{homogeneous term estimate})$ that
\[
\|\scal{\nabla}^\alpha F'(u)\|_{L^{\frac{2(2+\varepsilon)}{\varepsilon}}_t L^{\frac{d(2+\varepsilon)}{2\varepsilon+8}}_x} \lesssim (\mu^\theta Z_I^{1-\theta}+Z_I)^{\nu-1}.
\]
Thus
\[
\|(\Delta Iu) F'(u)-I(\Delta u F'(u))\|_{L^2_t L^{\frac{2d}{d+4}}_x} \lesssim N^{-\alpha}Z_I(\mu^\theta Z_I^{1-\theta}+Z_I)^{\nu-1},
\]
and
\begin{align}
|(\ref{almost conservation estimate 4})| \lesssim N^{-(2-\gamma+\delta)} Z_I^2(\mu^\theta Z_I^{1-\theta}+Z_I)^{\nu-1}. \label{almost conservation estimate 4 final}
\end{align}
Similarly, 
\[
|(\ref{almost conservation estimate 5})| \lesssim \|\Delta I u\|_{L^2_t L^{\frac{2d}{d-4}}_x} \|(I\nabla u) \cdot \nabla u F''(u) - I(\nabla u \cdot \nabla u F''(u))\|_{L^2_t L^{\frac{2d}{d+4}}_x}.
\]
We next apply Lemma $\ref{lem product rule}$ with $q=\frac{2d}{d+4}, q_1=\frac{4kd(\nu-1)}{(kd-4(k-2))(\nu-1)+8k}$ and $q_2=\frac{4kd(\nu-1)}{(kd+12k-8)(\nu-1)-8k}$ to have
\begin{multline*}
\|(I\nabla u) \cdot \nabla u F''(u) - I(\nabla u \cdot \nabla u F''(u))\|_{L^{\frac{2d}{d+4}}_x} \lesssim N^{-\alpha} \|I\nabla u\|_{L^{\frac{4kd(\nu-1)}{(kd-4(k-2))(\nu-1)+8k}}_x} \\
\times \|\scal{\nabla}^\alpha (\nabla u F''(u))\|_{L^{\frac{4kd(\nu-1)}{(kd+12k-8)(\nu-1)-8k}}_x}.
\end{multline*}
Using the notation $(\ref{define q q_star r r_star})$, the fractional chain rule implies
\[
\|\scal{\nabla}^\alpha (\nabla u F''(u))\|_{L^{\frac{4kd(\nu-1)}{(kd+12k-8)(\nu-1)-8k}}_x} \lesssim \|\scal{\nabla}^{1+\alpha} u\|_{L^q_x} \|F''(u)\|_{L^r_x} +\|\nabla u\|_{L^q_x} \|\scal{\nabla}^\alpha F''(u)\|_{L^r_x}.
\]
The H\"older inequality then gives
\begin{multline*}
\|(I\nabla u) \cdot \nabla u F''(u) - I(\nabla u \cdot \nabla u F''(u))\|_{L^2_t L^{\frac{2d}{d+4}}_x} \lesssim N^{-\alpha} \|I \nabla u\|_{L^{\frac{4k}{k-2}}_t L^q_x} \\
\times \Big( \|\scal{\nabla}^{1+\alpha} u \|_{L^{\frac{4k}{k-2}}_t L^q_x} \|F''(u)\|_{L^k_t L^r_x} + \|\nabla u\|_{L^{\frac{4k}{k-2}}_t L^q_x} \|\scal{\nabla}^\alpha F''(u)\|_{L^k_t L^r_x}\Big).
\end{multline*}
By the Sobolev embedding (dropping the $I$-operator if necessary) and $(\ref{property 3})$, we have 
\begin{align}
\|I \nabla u\|_{L^\frac{4k}{k-2}_t L^q_x}, \|\nabla u\|_{L^{\frac{4k}{k-2}}_t L^q_x} &\lesssim \||\nabla|^{1+\frac{\Gact}{2}} u\|_{L^{\frac{4k}{k-2}}_t L^{q^\star}_x} \lesssim Z_I, \label{estimate first derivative 1} \\
\|\scal{\nabla}^{1+\alpha} u \|_{L^{\frac{4k}{k-2}}_t L^q_x} &\lesssim \|\scal{\nabla}^{1+\alpha+\frac{\Gact}{2}} u\|_{L^{\frac{4k}{k-2}}_t L^{q^\star}_x} \lesssim Z_I. \label{estimate first derivative 2}
\end{align}
Note that by our assumptions on $\delta$, $1+\alpha+\frac{\Gact}{2}=3-\gamma+\delta+\frac{\Gact}{2}<\gamma$. We also have
\begin{align}
\|F''(u)\|_{L^k_t L^r_x} \lesssim \|u\|^{\nu-2}_{L^{k(\nu-2)}_t L^{r(\nu-2)}_x} \lesssim \||\nabla|^{\Gact} u\|^{\nu-2}_{L^{k(\nu-2)}_t L^{r^\star}_x} \lesssim Z_I^{\nu-2}. \label{estimate second derivative}
\end{align}
It remains to treat $\|\scal{\nabla}^\alpha F''(u)\|_{L^k_t L^r_x}$. 
Using $(\ref{estimate second derivative})$, we only need to bound $\||\nabla|^\alpha F''(u)\|_{L^k_t L^r_x}$. To do so, we separate two cases: $1\leq \nu-2$ and $0<\nu-2<1$. \newline
\indent If $1\leq \nu-2$, then we apply Lemma $\ref{lem fractional chain rule C1}$ for $q=r, q_1 = \frac{r(\nu-2)}{\nu-3}, q_2=r(\nu-2)$ and use H\"older's inequality to have
\begin{align*}
\||\nabla|^\alpha F''(u)\|_{L^k_t L^r_x} &\lesssim \|O(|u|^{\nu-3})\|_{L^{\frac{k(\nu-2)}{\nu-3}}_t L^{\frac{r(\nu-2)}{\nu-3}}_x} \||\nabla|^\alpha u\|_{L^{k(\nu-2)}_t L^{r(\nu-2)}_x} \\
&\lesssim \|u\|^{\nu-3}_{L^{k(\nu-2)}_t L^{r(\nu-2)}_x} \||\nabla|^\alpha u \|_{L^{k(\nu-2)}_t L^{r(\nu-2)}_x} \\
&\lesssim \||\nabla|^\Gact u\|^{\nu-3}_{L^{k(\nu-2)}_t L^{r^\star}_x} \||\nabla|^{\alpha+\Gact} u\|_{L^{k(\nu-2)}_t L^{r^\star}_x} \\
&\lesssim Z_I^{\nu-2}.
\end{align*}
Here by our assumptions, $\alpha +\Gact<\gamma$ which allows us to use $(\ref{property 3})$ to get the last estimate. \newline
\indent If $0<\nu-2<1$, then we use Lemma $\ref{lem fractional chain rule holder continuous}$ with $\beta=\nu-2, \alpha=2-\gamma+\delta, q=r$ and $q_1, q_2$ satisfying
\[
\Big(\nu-2-\frac{\alpha}{\rho}\Big)q_1=\frac{\alpha}{\rho}q_2 = r(\nu-2),
\]
and $\frac{\alpha}{\nu-2}<\rho<1$ to be chosen later. With these choices, we have
\[
\Big(1-\frac{\alpha}{\beta\rho}\Big)q_1=r>1.
\]
Then,
\[
\||\nabla|^\alpha F''(u)\|_{L^r_x} \lesssim \||u|^{\nu-2-\frac{\alpha}{\rho}} \|_{L^{q_1}_x} \||\nabla|^\rho u\|^{\frac{\alpha}{\rho}}_{L^{\frac{\alpha}{\rho}q_2}_x} \lesssim \|u\|^{\nu-2-\frac{\alpha}{\rho}}_{L^{\left(\nu-2-\frac{\alpha}{\rho}\right)q_1}_x} \||\nabla|^\rho u\|^{\frac{\alpha}{\rho}}_{L^{\frac{\alpha}{\rho}}_x}.
\]
By H\"older's inequality,
\begin{align*}
\||\nabla|^\alpha F''(u)\|_{L^k_t L^r_x} &\lesssim \|u\|^{\nu-2-\frac{\alpha}{\rho}}_{L^{\left(\nu-2-\frac{\alpha}{\rho}\right)p_1}_t L^{\left(\nu-2-\frac{\alpha}{\rho}\right)q_1}_x} \||\nabla|^\rho u\|^{\frac{\alpha}{\rho}}_{L^{\frac{\alpha}{\rho} p_2}_t L^{\frac{\alpha}{\rho}q_2}_x} \\
&\lesssim \|u\|^{\nu-2-\frac{\alpha}{\rho}}_{L^{k(\nu-2)}_t L^{r(\nu-2)}_x} \||\nabla|^\rho u\|_{L^{k(\nu-2)}_t L^{r(\nu-2)}_x},
\end{align*}
provided that
\[
\Big(\nu-2 -\frac{\alpha}{\rho}\Big) p_1 = \frac{\alpha}{\rho} p_2 = k(\nu-2).
\]
The Sobolev imbedding then gives
\[
\||\nabla|^\alpha F''(u)\|_{L^k_t L^r_x} \lesssim \||\nabla|^{\Gact}u\|^{\nu-2-\frac{\alpha}{\rho}}_{L^{k(\nu-2)}_t L^{r^\star}_x} \||\nabla|^{\rho+\Gact} u\|_{L^{k(\nu-2)}_t L^{r^\star}_x} \lesssim Z_I^{\nu-2}.
\]
Here we use $(\ref{property 3})$ together with $\rho+\Gact<\gamma$ to get the last estimate. Note that 
\[
\frac{\alpha}{\nu-2}+\Gact<\rho+\Gact. 
\] 
If we want $\rho+\Gact<\gamma$ for an appropriate value of $\rho$, we need $\frac{\alpha}{\nu-2}+\Gact<\gamma$. This implies $\gamma>\frac{2}{\nu-1}+\frac{\nu-2}{\nu-1}\Gact$ and $\delta <(\nu-1)\gamma-2-(\nu-2)\Gact$. Collecting 2 cases, we show 
\begin{align}
\||\nabla|^\alpha F''(u)\|_{L^k_t L^r_x} \lesssim Z_I^{\nu-2}. \label{estimate homogeneous second derivative}
\end{align}
By $(\ref{estimate first derivative 1}), (\ref{estimate first derivative 2}), (\ref{estimate second derivative})$ and $(\ref{estimate homogeneous second derivative})$, 
\[
\|(I\nabla u) \cdot \nabla u F''(u)-I(\nabla u \cdot \nabla u F''(u))\|_{L^2_t L^{\frac{2d}{d+4}}_x} \lesssim N^{-\alpha} Z_I^\nu.
\]
Thus,
\begin{align}
|(\ref{almost conservation estimate 5})| \lesssim N^{-(2-\gamma+\delta)} Z_I^{\nu+1}. \label{almost conservation estimate 5 final}
\end{align}
Finally, we consider $(\ref{almost conservation estimate 6})$. We bound
\begin{align}
|(\ref{almost conservation estimate 6})| &\lesssim \||\nabla|^{-1} IF(u)\|_{L^2_tL^{\frac{2d}{d-2}}_x} \|\nabla(F(Iu)-IF(u))\|_{L^2_t L^{\frac{2d}{d+2}}_x} \nonumber \\
&\lesssim \|\nabla I F(u)\|_{L^2_t L^{\frac{2d}{d+2}}_x} \|\nabla(F(Iu)-IF(u))\|_{L^2_t L^{\frac{2d}{d+2}}_x}.  \label{almost conservation estimate 6 sub 1}
\end{align}
By $(\ref{commutator estimate 2})$, 
\begin{align}
\|\nabla I F(u)\|_{L^2_t L^{\frac{2d}{d+2}}_x} \lesssim N^{-(2-\gamma+\delta)}Z_I^{\nu} + \mu^{(\nu-1)\theta}Z_I^{1+(\nu-1)(1-\theta)}. \label{almost conservation estimate 6 sub 2}
\end{align}
By the triangle inequality, 
\[
\|\nabla(F(Iu)-IF(u))\|_{L^2_t L^{\frac{2d}{d+2}}_x} \lesssim \|(\nabla I u) (F'(Iu)-F'(u))\|_{L^2_t L^{\frac{2d}{d+2}}_x} + \|(\nabla I u) F'(u)- \nabla I F(u)\|_{L^2_t L^{\frac{2d}{d+2}}_x}.
\]
We firstly use H\"older's inequality and  estimate as in $(\ref{almost conservation estimate 1 sub 1})$ to get
\begin{align}
\|(\nabla I u)(F'(Iu)-F'(u))\|_{L^2_t L^{\frac{2d}{d+2}}_x} &\lesssim \|\nabla I u\|_{L^{2+\varepsilon}_t L^{\frac{2d(2+\varepsilon)}{d(2+\varepsilon)-12}}_x} \|F'(Iu)-F'(u)\|_{L^{\frac{2(2+\varepsilon)}{\varepsilon}}_t L^{\frac{d(2+\varepsilon)}{2\varepsilon+8}}_x} \nonumber \\
&\lesssim \|\Delta I u\|_{L^{2+\varepsilon}_t L^{\frac{2d(2+\varepsilon)}{d(2+\varepsilon)-8}}_x} \|F'(Iu)-F'(u)\|_{L^{\frac{2(2+\varepsilon)}{\varepsilon}}_t L^{\frac{d(2+\varepsilon)}{2\varepsilon+8}}_x} \nonumber \\
&\lesssim N^{\Gact-2}Z_I^2(\mu^\theta Z_I^{1-\theta} +Z_I)^{\nu-2}. \label{almost conservation estimate 6 sub 3}
\end{align}
By $(\ref{commutator estimate 1})$, 
\begin{align}
\|(\nabla I u) F'(u)- \nabla I F(u)\|_{L^2_t L^{\frac{2d}{d+2}}_x} \lesssim N^{-(2-\gamma+\delta)} Z_I (\mu^\theta Z_I^{1-\theta} + Z_I)^{\nu-1}. \label{almost conservation estimate 6 sub 4}
\end{align}
Combining $(\ref{almost conservation estimate 6 sub 1})-(\ref{almost conservation estimate 6 sub 4})$, we get
\begin{align}
|(\ref{almost conservation estimate 6})| \lesssim N^{-(2-\gamma+\delta)}Z_I^2 (\mu^\theta Z_I^{1-\theta}+ Z_I)^{2(\nu-1)}. \label{almost conservation estimate 6 final}
\end{align}
Collecting $(\ref{almost conservation estimate 1 final}), (\ref{almost conservation estimate 2 final}), (\ref{almost conservation estimate 3 final}), (\ref{almost conservation estimate 4 final}), (\ref{almost conservation estimate 5 final}), (\ref{almost conservation estimate 6 final})$ and using $(\ref{control Z_I norm})$, we prove $(\ref{almost conservation law})$. Note that our assumptions on $\delta$ implies
\[
2-\gamma+\delta< \min\left\{\gamma-1-\frac{\Gact}{2}, \nu-2, (\nu-2)(\gamma-\Gact) \right\}< \min\left\{\frac{2-\Gact}{2}, (\nu-2)(2-\Gact)\right\}.
\]
The proof is complete.
\defendproof
\section{Global well-posedness and scattering} \label{section global well-posedness scattering}
\setcounter{equation}{0}
In this section, we shall give the proof of the global existence and scattering given in Theorem $\ref{theorem global existence scattering}$. 
\subsection*{Global well-posedness} 
By the density argument, the proof of global well-posedness will be reduced to the following.
\begin{prop} \label{prop uniform bound}
Let $5\leq d\leq 11$ and $\gamma(d,\nu)<\gamma<2$ with $\gamma(d,\nu)$ be as in $(\ref{define gamma d nu})$. Suppose that $u$ is a global solution to \emph{(NL4S)} with initial data $u_0 \in C^\infty_0(\R^d)$. Then,
\begin{align}
\|u\|_{M^\sigma(\R)} &\leq C(\|u_0\|_{H^\gamma_x}), \label{global bound morawetz norm} \\
\|u\|_{L^\infty_t(\R, H^\gamma_x)} &\leq C(\|u_0\|_{H^\gamma_x}), \label{global bound sobolev norm}
\end{align}
where $\|\cdot\|_{M^\sigma}$ is given in $(\ref{define M sigma})$. 
\end{prop}
\begin{proof}
The proof of this result is based on the almost conservation law given in Proposition $\ref{prop almost conservation law}$. To do so, we need the modified energy of initial data is small. Since $E(Iu_0)$ is not necessarily small, we use the scaling $(\ref{scaling definition})$ to make $E(Iu_\lambda(0))$ small. We have
\begin{align}
E(Iu_\lambda(0))=\frac{1}{2}\|Iu_\lambda (0)\|^2_{\dot{H}^2_x} +\frac{1}{\nu+1} \|Iu_\lambda (0)\|^{\nu+1}_{L^{\nu+1}_x}. \label{scaling modified energy}
\end{align}
By $(\ref{property 5})$,
\begin{align}
\|Iu_\lambda(0)\|_{\dot{H}^2_x} \lesssim N^{2-\gamma}\|u_\lambda(0)\|_{\dot{H}^\gamma_x}= N^{2-\gamma} \lambda^{\Gact-\gamma} \|u_0\|_{\dot{H}^\gamma_x}. \label{scaling homogeneous norm}
\end{align}
By choosing 
\begin{align}
\lambda \approx N^{\frac{2-\gamma}{\gamma-\Gact}}, \label{choice of lambda}
\end{align}
we have $\|Iu_\lambda(0)\|_{\dot{H}^2_x} \leq \frac{1}{8}$. We next bound $\|I u_\lambda(0)\|_{L^{\nu+1}_x}$. Note that we can easily estimate this norm by the Sobolev embedding,
\[
\|I u_\lambda(0)\|_{L^{\nu+1}_x} \lesssim \|u_\lambda(0)\|_{L^{\nu+1}_x} =\lambda^{\frac{d}{\nu+1}-\frac{4}{\nu-1}}\|u_0\|_{L^{\nu+1}_x} \lesssim \lambda^{\frac{(d-4)(\nu-1)-8}{(\nu+1)(\nu-1)}} \|u_0\|_{H^\gamma_x},
\] 
provided that $\gamma\geq \frac{d(\nu-1)}{2(\nu+1)}$. In order to remove this unexpected condition on $\gamma$, we use the technique of \cite{I-teamglobal} (see also \cite{MiaoWuZhang}). We firstly separate the frequency space into the domains 
\[
\Omega_1:=\Big\{\xi \in \R^d, \  |\xi| \lesssim \frac{1}{\lambda}\Big\}, \quad \Omega_2:=\Big\{ \xi \in \R^d, \ \frac{1}{\lambda} \lesssim |\xi| \lesssim N \Big\}, \quad \Omega_3:= \Big\{ \xi\in \R^d, \ |\xi| \gtrsim N \Big\},
\]
and then write
\[
[u_\lambda(0)]\widehat{\ }(\xi)=(\chi_1(\xi) + \chi_2(\xi) + \chi_3(\xi)) [u_\lambda(0)]\widehat{\ } (\xi),
\] 
for non-negative smooth functions $\chi_j$ supported in $\Omega_j, j=1,2,3$ respectively and satisfying $\sum \chi_j(\xi)=1$. Thus
\[
I u_\lambda(0) = \chi_1(D) I u_\lambda(0) + \chi_2(D) I u_\lambda(0) + \chi_3(D) I u_\lambda(0).
\]
We now use the Sobolev embedding to have
\[
\|\chi_1(D)I u_\lambda(0)\|_{L^{\nu+1}_x} \lesssim \||\nabla|^{\frac{d(\nu-1)}{2(\nu+1)}} \chi_1(D) I u_\lambda(0)\|_{L^2_x} \lesssim \||\nabla|^{\frac{d(\nu-1)}{2(\nu+1)}} \chi_1(D) I\|_{L^2_x \rightarrow L^2_x}\|u_\lambda(0)\|_{L^2_x}.
\]
Thanks to the support of $\chi_1$, the functional calculus gives 
\begin{align}
\||\nabla|^{\frac{d(\nu-1)}{2(\nu+1)}} \chi_1(D) I\|_{L^2_x \rightarrow L^2_x} \lesssim \||\xi|^{\frac{d(\nu-1)}{2(\nu+1)}-\alpha} |\xi|^\alpha \chi_1(\xi)\|_{L^\infty_\xi} \lesssim \lambda^{\alpha-\frac{d(\nu-1)}{2(\nu+1)}}, \label{norm estimate 1}
\end{align}
provided $0<\alpha <\frac{d(\nu-1)}{2(\nu+1)}$. Similarly,
\[
\|\chi_3(D) I u_\lambda(0)\|_{L^{\nu+1}_x} \lesssim \||\nabla|^{\frac{d(\nu-1)}{2(\nu+1)}} \chi_3(D) I u_\lambda(0)\|_{L^2_x} \lesssim \||\nabla|^{\frac{d(\nu-1)}{2(\nu+1)}-\gamma} \chi_3(D) I\|_{L^2_x \rightarrow L^2_x} \|u_\lambda(0)\|_{\dot{H}^\gamma_x}.
\]
A direct computation shows
\begin{align}
\|u_\lambda(0)\|_{\dot{H}^\gamma_x} = \lambda^{-\gamma}\|u_0\|_{\dot{H}^\gamma_x}. \label{H gamma norm}
\end{align}
Using the support of $\chi_3$, the functional calculus again gives
\begin{align}
\||\nabla|^{\frac{d(\nu-1)}{2(\nu+1)}-\gamma} \chi_3(D) I\|_{L^2_x \rightarrow L^2_x} \lesssim \||\xi|^{\frac{d(\nu-1)}{2(\nu+1)}-\gamma}\chi_3(\xi) (N|\xi|^{-1})^{2-\gamma}\|_{L^\infty_\xi} \lesssim N^{\frac{d(\nu-1)}{2(\nu+1)}-\gamma}. \label{estimate chi 3}
\end{align}
To obtain this bound, we split into two cases. \newline
\indent When $\frac{d(\nu-1)}{2(\nu+1)} \geq \gamma$, we simply bound
\[
\||\xi|^{\frac{d(\nu-1)}{2(\nu+1)}-\gamma}\chi_3(\xi) (N|\xi|^{-1})^{2-\gamma}\|_{L^\infty_\xi} \lesssim 1 \lesssim N^{\frac{d(\nu-1)}{2(\nu+1)}-\gamma}.
\]
\indent When $\gamma> \frac{d(\nu-1)}{2(\nu+1)}$, we write
\begin{align*}
\||\xi|^{\frac{d(\nu-1)}{2(\nu+1)}-\gamma}\chi_3(\xi) (N|\xi|^{-1})^{2-\gamma}\|_{L^\infty_\xi} &= N^{\frac{d(\nu-1)}{2(\nu+1)}-\gamma} \|(N|\xi|^{-1})^{\gamma-\frac{d(\nu-1)}{2(\nu+1)}} \chi_3(\xi) (N|\xi|^{-1})^{2-\gamma}\|_{L^\infty_\xi} \\
&	\lesssim N^{\frac{d(\nu-1)}{2(\nu+1)}-\gamma}.
\end{align*}
Combining $(\ref{H gamma norm})$ and $(\ref{estimate chi 3})$, we get
\begin{align}
\|\chi_3(D) I u_\lambda(0)\|_{L^{\nu+1}_x} \lesssim N^{\frac{d(\nu-1)}{2(\nu+1)}-\gamma}\lambda^{-\gamma}\|u_0\|_{\dot{H}^\gamma_x}. \label{norm estimate 2}
\end{align}
We treat the intermediate case as
\[
\|\chi_2(D) I u_\lambda(0)\|_{L^{\nu+1}_x} \lesssim \||\nabla|^{\frac{d(\nu-1)}{2(\nu+1)}-\gamma} \chi_2(D) I\|_{L^2_x \rightarrow L^2_x} \|u_\lambda(0)\|_{\dot{H}^\gamma_x}.
\] 
We have
\[
\||\nabla|^{\frac{d(\nu-1)}{2(\nu+1)}-\gamma} \chi_2(D) I\|_{L^2_x \rightarrow L^2_x} \lesssim \|\xi|^{\frac{d(\nu-1)}{2(\nu+1)}-\gamma} \chi_2(\xi)\|_{L^\infty_\xi}. 
\]
\indent When $\frac{d(\nu-1)}{2(\nu+1)}\geq \gamma$, we bound
\[
\|\xi|^{\frac{d(\nu-1)}{2(\nu+1)}-\gamma} \chi_2(\xi)\|_{L^\infty_\xi} \lesssim N^{\frac{d(\nu-1)}{2(\nu+1)}-\gamma}.
\]
\indent When $\gamma>\frac{d(\nu-1)}{2(\nu+1)}$, we write
\[
\|\xi|^{\frac{d(\nu-1)}{2(\nu+1)}-\gamma} \chi_2(\xi)\|_{L^\infty_\xi} = \|\xi|^{\frac{d(\nu-1)}{2(\nu+1)}-\gamma-\beta} |\xi|^\beta \chi_2(\xi)\|_{L^\infty_\xi} \lesssim \lambda^{\beta+\gamma-\frac{d(\nu-1)}{2(\nu+1)}},
\]
provided $\frac{d(\nu-1)}{2(\nu+1)}-\gamma <\beta <\frac{d(\nu-1)}{2(\nu+1)}$. These estimates together with $(\ref{H gamma norm})$ yield
\begin{align}
\|\chi_2(D) I u_\lambda(0)\|_{L^{\nu+1}_x} \lesssim \left\{ \begin{array}{l l}
N^{\frac{d(\nu-1)}{2(\nu+1)}-\gamma} \lambda^{-\gamma} \|u_0\|_{\dot{H}^\gamma_x} & \text{ if } \frac{d(\nu-1)}{2(\nu+1)} \geq \gamma, \\
\lambda^{\beta -\frac{d(\nu-1)}{2(\nu+1)}} \|u_0\|_{\dot{H}^\gamma_x} & \text{ if } \gamma>\frac{d(\nu-1)}{2(\nu+1)}.
\end{array} \right. \label{norm estimate 3}
\end{align}
Collecting $(\ref{norm estimate 1}), (\ref{norm estimate 2}), (\ref{norm estimate 3})$ and use $(\ref{choice of lambda})$, we obtain
\begin{align}
\|Iu_\lambda(0)\|_{L^{\nu+1}_x} \lesssim \Big(\lambda^{\alpha-\frac{d(\nu-1)}{2(\nu+1)}} + \lambda^{\beta -\frac{d(\nu-1)}{2(\nu+1)}} + \lambda^{\frac{((d-4)(\nu-1)-8)\gamma}{2(2-\gamma)(\nu+1)}}\Big) \|u_0\|_{H^\gamma_x}, \label{scaling inhomogeneous norm}
\end{align}
for some $0<\alpha<\frac{d(\nu-1)}{2(\nu+1)}$ and $\frac{d(\nu-1)}{2(\nu+1)}-\gamma <\beta<\frac{d(\nu-1)}{2(\nu+1)}$. Therefore, it follows from $(\ref{scaling homogeneous norm}), (\ref{choice of lambda})$ and $(\ref{scaling inhomogeneous norm})$ by taking $\lambda$ sufficiently large depending on $\|u_0\|_{H^\gamma_x}$ and $N$ (which will be chosen later and depend only on $\|u_0\|_{H^\gamma_x}$) that
\[
E(Iu_\lambda(0)) \leq \frac{1}{4}.
\] 
We now show that there exists an absolute constant $C$ such that
\begin{align}
\|u_\lambda\|_{M^\sigma(\R)} \leq C \lambda^{\Gact+\frac{\sigma(4-d)\Gact}{2(d-5+4\sigma)}}. \label{scaling morawetz norm bound}
\end{align}
By undoing the scaling, using the fact that
\[
\|u_\lambda\|_{M^\sigma(\R)} = \lambda^{\Gact+\frac{\sigma(4-d)}{d-5+4\sigma}} \|u\|_{M^\sigma(\R)},
\]
we get $(\ref{global bound morawetz norm})$. We shall use the bootstrap argument to show $(\ref{scaling morawetz norm bound})$. By time reversal symmetry, it suffices to treat the positive time only. To do so, we define
\[
\Omega_1:= \left\{ t \in [0,\infty) \ | \  \|u_\lambda\|_{M^\sigma([0,t])} \leq C \lambda^{\Gact+\frac{\sigma(4-d)\Gact}{2(d-5+4\sigma)}}  \right\}.
\]
We want to show $\Omega_1 =[0,\infty)$. Let
\[
\Omega_2:= \left\{ t \in [0,\infty) \ | \ \|u_\lambda\|_{M^\sigma([0,t])} \leq 2 C \lambda^{\Gact+\frac{\sigma(4-d)\Gact}{2(d-5+4\sigma)}}   \right\}.
\]
In order to run the bootstrap argument successfully, we need to verify four things:
\begin{itemize}
\item[1)] $\Omega_1 \ne \emptyset$. This is obvious as $0\in \Omega_1$. 
\item[2)] $\Omega_1$ is closed. This follows from Fatou's Lemma.
\item[3)] $\Omega_2 \subset \Omega_1$.
\item[4)] If $T\in \Omega_1$, then there exists $\delta>0$ such that $[T, T+\delta) \subset \Omega_2$. This is a consequence of the local well-posedness and 3). 
\end{itemize} 
It remains to prove 3). Fix $T\in \Omega_2$, we will show that $T\in \Omega_1$. We firstly use the interaction Morawetz inequality $(\ref{interaction morawetz interpolation})$ and the mass conservation to have
\begin{align}
\|u_\lambda\|_{M^\sigma([0,T])} &\lesssim \Big(\|u_\lambda(0)\|_{L^2_x} \|u_\lambda\|_{L^\infty_t([0,T],\dot{H}^{\frac{1}{2}}_x)}\Big)^{\frac{2\sigma}{d-5+4\sigma}} \|u_\lambda\|_{L^\infty_t([0,T],\dot{H}^\sigma_x)}^{\frac{d-5}{d-5+4\sigma}} \nonumber \\
&\lesssim C(\|u_0\|_{L^2_x}) \lambda^{\frac{2\sigma \Gact}{d-5+4\sigma}} \|u_\lambda\|^{\frac{2\sigma}{d-5+4\sigma}}_{L^\infty_t([0,T],\dot{H}^{\frac{1}{2}}_x)} \|u_\lambda\|^{\frac{d-5}{d-5+4\sigma}}_{L^\infty_t([0,T], \dot{H}^\sigma_x)}. \label{scaling morawetz estimate}
\end{align}
We now decompose 
\[
u_\lambda(t) := P_{\leq N} u_\lambda(t) + P_{>N} u_\lambda(t)
\]
to estimate the second and the third factor in the right hand side of $(\ref{scaling morawetz estimate})$. For the low frequency part, we interpolate between the $L^2_x$-norm and $\dot{H}^2_x$-norm to have
\begin{align}
\|P_{\leq N} u_\lambda(t)\|_{\dot{H}^{\frac{1}{2}}_x}  &\lesssim \|P_{\leq N}  u_\lambda (t)\|^{\frac{3}{4}}_{L^2_x} \|P_{\leq N} u_\lambda (t)\|^{\frac{1}{4}}_{\dot{H}^2_x} \lesssim C(\|u_0\|_{L^2_x})\lambda^{\frac{3\Gact}{4}} \|I u_\lambda(t)\|^{\frac{1}{4}}_{\dot{H}^2_x}, \label{low frequency estimate 1} \\
\|P_{\leq N} u_\lambda(t)\|_{\dot{H}^{\sigma}_x}  &\lesssim \|P_{\leq N}  u_\lambda (t)\|^{1-\frac{\sigma}{2}}_{L^2_x} \|P_{\leq N} u_\lambda (t)\|^{\frac{\sigma}{2}}_{\dot{H}^2_x} \lesssim C(\|u_0\|_{L^2_x}) \lambda^{\frac{\Gact(2-\sigma)}{2}} \|I u_\lambda(t)\|^{\frac{\sigma}{2}}_{\dot{H}^2_x}. \label{low frequency estimate 2}
\end{align}
Note that the $I$-operator is the identity on low frequency $|\xi|\leq N$. For high frequency part, we interpolate between the $L^2_x$-norm and $\dot{H}^\gamma_x$-norm and use $(\ref{property 2})$ to have
\begin{align}
\|P_{> N} u_\lambda(t)\|_{\dot{H}^{\frac{1}{2}}_x}  &\lesssim \|P_{> N}  u_\lambda (t)\|^{1-\frac{1}{2\gamma}}_{L^2_x} \|P_{> N} u_\lambda (t)\|^{\frac{1}{2\gamma}}_{\dot{H}^\gamma_x} \nonumber \\
&\lesssim C(\|u_0\|_{L^2_x}) \lambda^{\Gact\left(1-\frac{1}{2\gamma}\right)} N^{\frac{\gamma-2}{2\gamma}} \|I u_\lambda(t)\|^{\frac{1}{2\gamma}}_{\dot{H}^2_x} \nonumber \\
&\lesssim C(\|u_0\|_{L^2_x}) \lambda^{\frac{3\Gact}{4}} \|I u_\lambda(t)\|^{\frac{1}{2\gamma}}_{\dot{H}^2_x}, \label{high frequency estimate 1} \\
\|P_{> N} u_\lambda(t)\|_{\dot{H}^{\sigma}_x}  &\lesssim \|P_{> N}  u_\lambda (t)\|^{1-\frac{\sigma}{\gamma}}_{L^2_x} \|P_{> N} u_\lambda (t)\|^{\frac{\sigma}{\gamma}}_{\dot{H}^2_x} \nonumber \\
&\lesssim C(\|u_0\|_{L^2_x}) \lambda^{\Gact\left(1-\frac{\sigma}{\gamma}\right)} N^{\frac{\sigma(\gamma-2)}{\gamma}} \|Iu_\lambda(t)\|^{\frac{\sigma}{\gamma}}_{\dot{H}^2_x} \nonumber \\
&\lesssim C(\|u_0\|_{L^2_x})\lambda^{\frac{\Gact(2-\sigma)}{2}} \|I u_\lambda(t)\|^{\frac{\sigma}{\gamma}}_{\dot{H}^2_x}. \label{high frequency estimate 2}
\end{align}
Here we use the fact $0<\gamma<2$ to get $(\ref{high frequency estimate 1})$ and $(\ref{high frequency estimate 2})$. Collecting $(\ref{scaling morawetz estimate})$ through $(\ref{high frequency estimate 2})$, we get
\begin{multline}
\|u_\lambda\|_{M^\sigma([0,T])} \lesssim C(\|u_0\|_{L^2_x}) \lambda^{\Gact+\frac{\sigma(4-d)\Gact}{2(d-5+4\sigma)}} \sup_{[0,T]}\left(\|Iu_\lambda(t)\|^{\frac{1}{4}}_{\dot{H}^2_x} + \|Iu_\lambda(t)\|^{\frac{1}{2\gamma}}_{\dot{H}^2_x}\right)^{\frac{2\sigma}{d-5+4\sigma}} \\
\times \sup_{[0,T]} \left(\|Iu_\lambda(t)\|^{\frac{\sigma}{2}}_{\dot{H}^2_x} + \|Iu_\lambda(t)\|^{\frac{\sigma}{\gamma}}_{\dot{H}^2_x}\right)^{\frac{d-5}{d-5+4\sigma}}. \label{scaling morawetz estimate argument}
\end{multline}
Thus, by taking $C$ sufficiently large depending on $\|u_0\|_{L^2_x}$, we get $T\in \Omega_1$, provided that 
\begin{align}
\sup_{[0,T]} \|Iu_\lambda(t)\|_{\dot{H}^2_x} \lesssim 1. \label{small condition}
\end{align}
We will prove that $(\ref{small condition})$ holds for $T\in \Omega_2$. Indeed, let $\mu>0$ be a sufficiently small constant given in Proposition $\ref{prop almost conservation law}$. We divide $[0,T]$ into subintervals $J_k, k=1,..., L$ in such a way that
\[
\|u_\lambda\|_{M^\sigma(J_k)} \leq \mu.
\]
The number of possible subinterval must satisfy
\begin{align}
L \sim \Big(\frac{\lambda^{\Gact+\frac{\sigma(4-d)\Gact}{2(d-5+4\sigma)}}}{\mu}\Big)^{\frac{d-5+4\sigma}{\sigma}} \sim \lambda^{\frac{\Gact(d-5+(8-d)\sigma)}{\sigma}}. \label{choice of L}
\end{align}
We next apply Proposition $\ref{prop almost conservation law}$ on each of the subintervals $J_k$ to have
\[
\sup_{[0,T]} \|I u_\lambda(t)\|^2_{L^2_x} \lesssim \sup_{[0,T]} E(Iu_\lambda(t)) \leq E(Iu_\lambda(0)) + C(E(Iu_\lambda(0))) N^{-(2-\gamma+\delta)} L.
\]
Since $E(Iu_\lambda(0)) \leq \frac{1}{4}$, we need
\begin{align}
N^{-(2-\gamma+\delta)} L \ll \frac{1}{4} \label{small requirement}
\end{align}
in order to guarantee $(\ref{small condition})$ holds. Combining $(\ref{choice of lambda}), (\ref{choice of L})$ and $(\ref{small requirement})$, we need to choose $N$ depending on $\|u_0\|_{H^\gamma_x}$ such that
\[
N^{\frac{\Gact(2-\gamma)(d-5+(8-d)\sigma)}{(\gamma-\Gact)\sigma} -(2-\gamma+\delta)} \ll 1.
\]
This is possible whenever $\gamma$ is such that
\[
\frac{\Gact(2-\gamma)(d-5+(8-d)\sigma)}{(\gamma-\Gact)\sigma} < 2-\gamma+\delta,
\]
or
\begin{align}
\Gact(2-\gamma) (d-5+(8-d)\sigma) < (2-\gamma+\delta)(\gamma-\Gact)\sigma. \label{condition gamma}
\end{align}
Since $\delta<\min \{2\gamma-3 - \frac{\Gact}{2}, \gamma+\nu-4, (\nu-1)\gamma-2-(\nu-2)\Gact\}$, we have $\gamma>\gamma(d,\nu,\sigma)$, where $\gamma(d,\nu,\sigma)$ is the (larger if there are two) root of the equation
\[
\Gact(2-\gamma)(d-5+(8-d)\sigma) = \min \left\{ \gamma-1-\frac{\Gact}{2}, \nu-2, (\nu-2) (\gamma-\Gact) \right\} (\gamma-\Gact)\sigma.
\]
This completes the bootstrap argument and $(\ref{scaling morawetz norm bound})$ follows. Thus, $(\ref{small condition})$ holds for all $T\in \R$.  \newline
\indent We now estimate $\|u(T)\|_{H^\gamma_x}$. To do so, we use the conservation of mass, the scaling $(\ref{scaling definition})$ and $(\ref{property 4})$ to have
\begin{align*}
\|u(T)\|_{H^\gamma_x} &\lesssim \|u(T)\|_{L^2_x}+ \|u(T)\|_{\dot{H}^\gamma_x} \\
&\lesssim \|u_0\|_{L^2_x} + \lambda^{\gamma-\Gact} \|u_\lambda(\lambda^4T)\|_{\dot{H}^\gamma_x} \\
&\lesssim \|u_0\|_{L^2_x} + \lambda^{\gamma-\Gact} \|Iu_\lambda(\lambda^4T)\|_{H^2_x} \\
&\lesssim \|u_0\|_{L^2_x} + \lambda^{\gamma-\Gact} \left(\|u_\lambda(\lambda^4T) \|_{L^2_x} + \|Iu_\lambda(\lambda^4T)\|_{\dot{H}^2_x}\right).
\end{align*}
Using $(\ref{small condition})$, we get for all $T\in \R$,
\[
\|u(T)\|_{H^\gamma_x} \lesssim \|u_0\|_{L^2_x} + \lambda^{\gamma-\Gact} (\lambda^\Gact \|u_0\|_{L^2_x}+1) \leq C(\|u_0\|_{H^\gamma_x}).
\]
Here we use $(\ref{choice of lambda})$ with the fact that $N$ is chosen sufficiently large depending only on $\|u_0\|_{H^\gamma_x}$. This proves $(\ref{global bound sobolev norm})$ and the proof of Proposition $\ref{prop uniform bound}$ is complete.
\end{proof}
\subsection*{Scattering} We firstly show that the global Morawetz estimate $(\ref{global bound morawetz norm})$ can be upgraded to the global Strichartz estimate
\begin{align}
\|u\|_{S^\gamma(\R)} := \sup_{(p,q)\in B} \|\scal{\nabla}^\gamma u\|_{L^p_t(\R, L^q_x)} \leq C(\|u_0\|_{H^\gamma_x}). \label{global bound strichartz norm}
\end{align}
Here we refer to Section 2 for the definition of $(p,q)\in B$. Let $u$ be a global solution to (NL4S) with initial data $u_0\in H^\gamma(\R^d)$ for $5\leq d\leq 11$ and $\gamma(d,\nu)<\gamma<2$. Using the uniform bound $(\ref{global bound morawetz norm})$, we can decompose $\R$ into a finite number of disjoint intervals $J_k=[t_k,t_{k+1}], k=1,...,L$ so that 
\begin{align}
\|u\|_{M^\sigma(J_k)} \leq \delta, \label{spliting morawetz norm}
\end{align}
for a small constant $\delta>0$ to be chosen later. By Strichartz estimates $(\ref{strichartz estimate biharmonic})$ and $(\ref{strichartz estimate biharmonic 4order})$, we have
\begin{align}
\|u\|_{S^\gamma(J_k)} \lesssim \|\scal{\nabla}^\gamma u(t_k)\|_{L^2_x} + \|F(u)\|_{L^2_t (J_k,L^{\frac{2d}{d+4}}_x)} + \||\nabla|^{\gamma-1} F(u)\|_{L^2_t(J_k, L^{\frac{2d}{d+2}}_x)}. \label{strichart 1}
\end{align}
We estimate for some $\varepsilon>0$,
\begin{align}
\||\nabla|^{\gamma-1} F(u)\|_{L^2_t(J_k,L^{\frac{2d}{d+2}}_x)} &\lesssim \||\nabla|^{\gamma-1}u\|_{L^{2+\varepsilon}_t(J_k,L^{\frac{2d(2+\varepsilon)}{(d-2)(2+\varepsilon)-8}}_x)} \|F'(u)\|_{L^{\frac{2(2+\varepsilon)}{\varepsilon}}_t(J_k, L^{\frac{d(2+\varepsilon)}{2\varepsilon+8}}_x)} \nonumber \\
&\lesssim \||\nabla|^{\gamma} u\|_{L^{2+\varepsilon}_t(J_k, L^{\frac{2d(2+\varepsilon)}{d(2+\varepsilon)-8}}_x) } \|u\|^{\nu-1}_{L^{\frac{2(\nu-1)(2+\varepsilon)}{\varepsilon}}_t(J_k, L^{\frac{d(\nu-1)(2+\varepsilon)}{2\varepsilon+8}}_x)} \nonumber \\
&\lesssim \|u\|_{S^\gamma(J_k)} \|u\|^{\nu-1}_{L^{\frac{2(\nu-1)(2+\varepsilon)}{\varepsilon}}_t(J_k, L^{\frac{d(\nu-1)(2+\varepsilon)}{2\varepsilon+8}}_x)}. \label{strichartz 2}
\end{align}
Similarly,
\begin{align}
\||F(u)\|_{L^2_t(J_k,L^{\frac{2d}{d+4}}_x)} &\lesssim \|u\|_{L^{2+\varepsilon}_t(J_k,L^{\frac{2d(2+\varepsilon)}{d(2+\varepsilon)-8}}_x)} \|F'(u)\|_{L^{\frac{2(2+\varepsilon)}{\varepsilon}}_t(J_k, L^{\frac{d(2+\varepsilon)}{2\varepsilon+8}}_x)} \nonumber \\
&\lesssim \|u\|_{L^{2+\varepsilon}_t(J_k, L^{\frac{2d(2+\varepsilon)}{d(2+\varepsilon)-8}}_x) } \|u\|^{\nu-1}_{L^{\frac{2(\nu-1)(2+\varepsilon)}{\varepsilon}}_t(J_k, L^{\frac{d(\nu-1)(2+\varepsilon)}{2\varepsilon+8}}_x)} \nonumber \\
&\lesssim \|u\|_{S^\gamma(J_k)} \|u\|^{\nu-1}_{L^{\frac{2(\nu-1)(2+\varepsilon)}{\varepsilon}}_t (J_k, L^{\frac{d(\nu-1)(2+\varepsilon)}{2\varepsilon+8}}_x)}. \label{strichartz 3}
\end{align}
We now need the following result.
\begin{lem} \label{lem scattering norm control}
Let $d\geq 5$, $0< \sigma\leq \gamma <2$ be such that $\frac{(d-4)\sigma}{d-5+4\sigma} <\gamma$ and $\frac{8}{d}<\nu-1<\frac{8}{d-2\gamma}$. Then there exists $\varepsilon>0$ small such that for any time interval $J$,
\begin{align}
\|u\|^{\nu-1}_{L^{\frac{2(\nu-1)(2+\varepsilon)}{\varepsilon}}_t(J, L^{\frac{d(\nu-1)(2+\varepsilon)}{2\varepsilon+8}}_x)} \lesssim \|u\|^{\frac{\varepsilon(d-5+4\sigma)}{2(2+\varepsilon)\sigma}}_{M^\sigma(J)} \|u\|^{\alpha(\varepsilon)}_{L^\infty_t(J, L^2_x)} \|u\|^{\beta(\varepsilon)}_{L^\infty_t(J, \dot{H}^\gamma_x)}, \label{scattering norm control}
\end{align}
where
\[
\alpha(\varepsilon):= \Big(1-\frac{d}{2\gamma}\Big)(\nu-1) + \frac{16\sigma+\varepsilon((d+4)\sigma-\gamma(d-5+4\sigma))}{2\gamma\sigma(2+\varepsilon)}, \quad \beta(\varepsilon):= \frac{d}{\gamma}\Big(\frac{\nu-1}{2}-\frac{16+\varepsilon(d+4)}{2d(2+\varepsilon)}\Big).
\]
\end{lem}
\begin{proof}
We firstly use H\"older's inequality to have
\begin{align}
\|u\|_{L^{\frac{2(\nu-1)(2+\varepsilon)}{\varepsilon}}_t(J, L^{\frac{d(\nu-1)(2+\varepsilon)}{2\varepsilon+8}}_x)} \leq \|u\|^{\theta_1}_{M^\sigma(J)} \|u\|^{1-\theta_1}_{L^\infty_t(J,L^q_x)}, \label{scattering 1}
\end{align}
provided that
\begin{align*}
\frac{1-\theta_1}{q}= \frac{4(2\varepsilon+8)\sigma-d\varepsilon(d-5+2\sigma)}{4d\sigma(\nu-1)(2+\varepsilon)} \quad \text{and}\quad \theta_1:= \frac{\varepsilon(d-5+4\sigma)}{2(\nu-1)(2+\varepsilon)\sigma}.\nonumber 
\end{align*}
Similarly, 
\begin{align}
\|u\|_{L^\infty_t(J, L^q_x)} \lesssim \|u\|^{\theta_2}_{L^\infty_t(J,L^2_x)} \|u\|^{1-\theta_2}_{L^\infty_t(J,L^{\frac{2d}{d-2\gamma}}_x)} \lesssim \|u\|^{\theta_2}_{L^\infty_t(J,L^2_x)} \|u\|^{1-\theta_2}_{L^\infty_t(J,\dot{H}^\gamma_x)}, \label{scattering 2}
\end{align}
provided that
\begin{align*}
\frac{1}{q}=\frac{\theta_2}{2}+\frac{(1-\theta_2)(d-2\gamma)}{2d}. 
\end{align*} 
Thus, by $(\ref{scattering 1})$ and $(\ref{scattering 2})$, a direct consequence gives
\[
\|u\|^{\nu-1}_{L^{\frac{2(\nu-1)(2+\varepsilon)}{\varepsilon}}_t(J, L^{\frac{d(\nu-1)(2+\varepsilon)}{2\varepsilon+8}}_x)} \lesssim \|u\|^{\frac{\varepsilon(d-5+4\sigma)}{2(2+\varepsilon)\sigma}}_{M^\sigma(J)} \|u\|_{L^\infty_t(J, L^2_x)}^{\alpha(\varepsilon)} \|u\|_{L^\infty_t(J,\dot{H}^\gamma_x)}^{\beta(\varepsilon)},
\]
where
\begin{align*}
\alpha(\varepsilon)&:=\theta_2(1-\theta_1)(\nu-1) = \Big(1-\frac{d}{2\gamma}\Big)(\nu-1) + \frac{16\sigma+\varepsilon((d+4)\sigma-\gamma(d-5+4\sigma))}{2\gamma\sigma(2+\varepsilon)}, \\
\beta(\varepsilon)&:=(1-\theta_2)(1-\theta_1)(\nu-1)=\frac{d}{\gamma}\Big(\frac{\nu-1}{2}-\frac{16+\varepsilon(d+4)}{2d(2+\varepsilon)}\Big).
\end{align*}
In order to perform the above estimates, we need $\alpha(\varepsilon)>0$ and $\beta(\varepsilon)>0$. We note that $\varepsilon \mapsto \alpha(\varepsilon)$ and $\varepsilon\mapsto \beta(\varepsilon)$ are decreasing functions provided that $\gamma>\frac{(d-4)\sigma}{d-5+4\sigma}$. Moreover, since
\[
\alpha(\varepsilon)\rightarrow \Big(1-\frac{d}{2\gamma}\Big)(\nu-1) +\frac{4}{\gamma}, \quad \beta(\varepsilon)\rightarrow \frac{d}{\gamma}\Big(\frac{\nu-1}{2}-\frac{4}{d}\Big) \text{ as } \varepsilon \rightarrow 0.
\]
As $\frac{8}{d}<\nu-1<\frac{8}{d-2\gamma}$, the two limits are positive. Thus by taking $\varepsilon>0$ small enough, we have $\alpha(\varepsilon)>0$ and $\beta(\varepsilon)>0$. The proof is complete.
\end{proof}
\begin{rem} \label{rem scattering}
It is easy to see that the function $\sigma \in (0,\gamma]\mapsto \frac{(d-4)\sigma}{d-5+4\sigma}$ is increasing and attains its maximal value at $\sigma=\gamma$. In this case, the condition $\frac{(d-4)\sigma}{d-5+4\sigma}<\gamma$ becomes $\gamma>\frac{1}{4}$ which is always satisfied in our consideration.
\end{rem}
We now continue the proof of scattering property. By $(\ref{strichart 1}), (\ref{strichartz 2}), (\ref{strichartz 3})$ and Lemma $\ref{lem scattering norm control}$, we have
\begin{align}
\|u\|_{S^\gamma(J_k)} &\lesssim \|\scal{\nabla}^\gamma u(t_k)\|_{L^2_x} + \|u\|_{S^\gamma(J_k)}\|u\|^{\frac{\varepsilon(d-5+4\sigma)}{2(2+\varepsilon)\sigma}}_{M^\sigma(J_k)} \|u\|^{\alpha(\varepsilon)}_{L^\infty_t(J_k,L^2_x)} \|u\|^{\beta(\varepsilon)}_{L^\infty_t(J_k, \dot{H}^\gamma_x)} \nonumber \\
&\lesssim \|\scal{\nabla}^\gamma u(t_k)\|_{L^2_x} + \|u\|_{S^\gamma(J_k)} \|u\|^{\frac{\varepsilon(d-5+4\sigma)}{2(2+\varepsilon)\sigma}}_{M^\sigma(J_k)} \|u\|^{\alpha(\varepsilon)+\beta(\varepsilon)}_{L^\infty_t(J_k,H^\gamma_x)}. \label{strichartz final}
\end{align}
This shows that
\[
\|u\|_{S^\gamma(J_k)} \lesssim \|\scal{\nabla}^\gamma u(t_k)\|_{L^2_x} + \|u\|_{S^\gamma(J_k)} \delta^{\frac{\varepsilon(d-5+4\sigma)}{2(2+\varepsilon)\sigma}} C(\|u_0\|_{H^\gamma_x}).
\]
By taking $\delta>0$ small enough, we get
\[
\|u\|_{S^\gamma(J_k)} \lesssim \|\scal{\nabla}^\gamma u(t_k)\|_{L^2_x} \leq C(\|u_0\|_{H^\gamma_x}).
\]
This proves $(\ref{global bound strichartz norm})$. \newline
\indent We next use the global Strichartz bound $(\ref{global bound strichartz norm})$ to prove the scattering property, i.e. there exist unique $u^\pm_0 \in H^\gamma_x$ such that 
\[
\lim_{t\rightarrow \pm \infty} \|u(t)-e^{it\Delta^2} u^\pm_0\|_{H^\gamma_x} =0.
\]
By the time reversal symmetry, it is enough to treat the positive time only. We will show that $e^{-it\Delta^2} u(t)$ has limits in $H^\gamma_x$ as $t\rightarrow +\infty$. By Duhamel formula,
\[
e^{-it\Delta^2}u(t) = u_0 +i\int_0^t e^{-is\Delta^2}F(u(s))ds.
\]
For $0<t_1<t_2$, we have
\[
e^{-it_2\Delta^2} u(t_2)-e^{-it_1\Delta^2}u(t_1) = i\int_{t_1}^{t_2} e^{-is\Delta^2} F(u(s))ds.
\]
By Strichartz estimates $(\ref{strichartz estimate biharmonic}), (\ref{strichartz estimate biharmonic 4order})$ and estimating as in $(\ref{strichartz final})$, 
\begin{align*}
\|e^{-it_2\Delta^2} u(t_2)-e^{-it_1\Delta^2}u(t_1)\|_{H^\gamma_x} &\lesssim \Big\| i\int_{t_1}^{t_2} e^{-is\Delta^2} F(u(s))ds\Big\|_{H^\gamma_x} \\
&\lesssim \|F(u)\|_{L^2_t([t_1,t_2],L^{\frac{2d}{d+4}}_x)} + \||\nabla|^{\gamma-1} F(u)\|_{L^2_t([t_1,t_2], L^{\frac{2d}{d+2}}_x)} \\
&\lesssim \|u\|_{S^\gamma([t_1,t_2])} \|u\|^{\frac{\varepsilon(d-5+4\sigma)}{2(2+\varepsilon)\sigma}}_{M^\sigma([t_1,t_2])} \|u\|^{\alpha(\varepsilon)+\beta(\varepsilon)}_{L^\infty_t([t_1,t_2],H^\gamma_x)} \\
&\lesssim \|u\|^{1+\alpha(\varepsilon)+\beta(\varepsilon)}_{S^\gamma([t_1,t_2])} \|u\|^{\frac{\varepsilon(d-5+4\sigma)}{2(2+\varepsilon)\sigma}}_{M^\sigma([t_1,t_2])}.
\end{align*}
This implies that $\|e^{-it_2\Delta^2} u(t_2)-e^{-it_1\Delta^2}u(t_1)\|_{H^\gamma_x} \rightarrow 0$ as $t_1, t_2 \rightarrow +\infty$. Hence the limit
\[
u^+_0:=\lim_{t\rightarrow +\infty} e^{-it\Delta^2}u(t)
\]
exists in $H^\gamma_x$. Moreover, 
\[
u(t)-e^{it\Delta^2}u^+_0=-i\int_t^{+\infty} e^{i(t-s)\Delta^2} F(u(s))ds.
\]
A same argument as above shows that
\[
\|u(t)-e^{it\Delta^2}u^+_0\|_{H^\gamma_x} \rightarrow 0
\]
as $t\rightarrow+\infty$. The proof is now complete.
\section*{Acknowledgments}
The author would like to express his deep thanks to his wife - Uyen Cong for her encouragement and support. He would like to thank his supervisor Prof. Jean-Marc Bouclet for the kind guidance and constant encouragement. He also would like to thank the reviewer for his/her helpful comments and suggestions.


\end{document}